\documentclass{amsart}

\makeatletter
 \def\@textbottom{\vskip \z@ \@plus 1pt}
 \let\@texttop\relax
\makeatother

\usepackage{a4wide}
\usepackage{tikz}
\usepackage{color,eucal,enumerate,mathrsfs}
\usepackage[normalem]{ulem}
\usepackage{amsmath}
\usepackage{amssymb,epsfig,bbm}
\usepackage[colorlinks,citecolor=green,linkcolor=red]{hyperref}
\usepackage{hyphenat}
\numberwithin{equation}{section}

\usepackage[latin1]{inputenc}

\newcommand{\N}{\mathbb{N}}

\newcommand{\R}{\mathbb{R}}
\newcommand{\Z}{\mathbb{Z}}

\newcommand{\mm}{{\mbox{\boldmath$m$}}}

\newcommand{\sfd}{{\sf d}}

\newcommand{\Kliminf}{K\kern-3pt-\kern-2pt\mathop{\rm lim\,inf}\limits}  
\renewcommand{\d}{{\mathrm d}}

\newcommand{\restr}[1]{\lower3pt\hbox{$|_{#1}$}}
\newcommand{\la}{{\langle}}                  
\newcommand{\ra}{{\rangle}}
\newcommand{\eps}{\varepsilon}  
\newcommand{\nchi}{{\raise.3ex\hbox{$\chi$}}}
\newcommand{\weakto}{\rightharpoonup}

\newcommand{\limi}{\varliminf}
\newcommand{\lims}{\varlimsup}

\newcommand{\fr}{\penalty-20\null\hfill$\blacksquare$}                      

\newcommand{\X}{{\rm X}}

\renewcommand{\mm}{\mathfrak m}                                

\renewcommand{\Cap}{{\rm Cap}}
\newcommand{\Cqc}{\mathcal{QC}(\X)}
\newcommand{\qu}{\mathcal{QU}}
\newcommand{\qcr}{{\sf QCR}}
\newcommand{\Cqcvf}{\mathcal{QC}(T\X)}
\newcommand{\tCqcvf}{{\widetilde{\mathcal{QC}}}(T\X)}
\newcommand{\RCD}{{\sf RCD}}
\newcommand{\ncRCD}{{\sf ncRCD}}

\newtheorem{theorem}{Theorem}[section]

\newtheorem{lemma}[theorem]{Lemma}
\newtheorem{proposition}[theorem]{Proposition}

\newtheorem{definition}[theorem]{Definition}

\newtheorem{example}[theorem]{Example}
\newtheorem{remark}[theorem]{Remark}

\newcommand{\beq}{\begin{equation}}
\newcommand{\eeq}{\end{equation}}

\title{Quasi-continuous vector fields on $\sf RCD$ spaces}

\author{Cl\'{e}ment Debin}
\address{Institut Polytechnique de Grenoble (CPP - La Pr\'epa des INP), Domaine universitaire, 701 rue de la piscine, 38402 Saint Martin d'H\`eres.}
\email{clement.debin@gmail.com}

\author{Nicola Gigli}
\address{SISSA, Via Bonomea 265, 34136 Trieste}
\email{ngigli@sissa.it}

\author{Enrico Pasqualetto}
\address{Department of Mathematics and Statistics,
P.O.\ Box 35 (MaD), FI-40014 University of Jyv\"{a}skyl\"{a}}
\email{enpasqua@jyu.fi}

\date{\today}

\begin{document}
\begin{abstract} In the existing language for tensor calculus on ${\sf RCD}$ spaces, tensor fields are only defined $\mm$-a.e.. In this paper we introduce the concept of tensor field defined `2-capacity-a.e.' and discuss in which sense Sobolev vector fields have a 2-capacity-a.e.\ uniquely defined quasi-continuous representative.
\end{abstract}
\maketitle
\tableofcontents
\section*{Introduction}

The theory of differential calculus on ${\sf RCD}$ spaces as proposed in \cite{Gigli14,Gigli17} is built around the notion of $L^0(\mm)$-normed $L^0(\mm)$-module, which provides a convenient abstraction of the concept of `space of measurable sections of a vector bundle'. In this sense, one thinks at such a module as the space of measurable sections of some, not really given, measurable bundle over the given metric measure space $(\X,\sfd,\mm)$. More precisely, given that elements of $L^0(\mm)$ are `Borel functions identified up to equality $\mm$-a.e.', elements of such modules are, in a sense, `measurable sections identified up to equality $\mm$-a.e.'. Notice that this interpretation is fully rigorous in the smooth case, where given a normed vector bundle, the space of its Borel sections identified up to $\mm$-a.e.\ equality is  a $L^0(\mm)$-normed $L^0(\mm)$-module. We also remark that in  \cite{Gigli14,Gigli17} the emphasis is more on the notion of $L^\infty(\mm)$-module rather than on $L^0(\mm)$ ones, but this is more a choice of presentation rather than an essential technical point, and given that for the purposes of this manuscript to work with $L^0$ is more convenient, we shall concentrate on this.

In particular, all the tensor fields on a metric measure space which are considered within the theory of $L^0(\mm)$-modules are only defined $\mm$-a.e.. While this is an advantage in some setting, e.g.\ because a rigorous first order differential calculus can be built on this ground over arbitrary metric measure structures, in others is quite limiting: there certainly might be instances where, say, one is interested in the behaviour of a vector field on some negligible set. For instance, the question of whether the critical set of a harmonic function has capacity zero simply makes no sense if the gradient of such function only exists as element of a $L^0(\mm)$-normed module.

Aim of this paper is to create a theoretical framework which allows to speak about `Borel sections identified up to equality ${\rm Cap}$-a.e.' and to show that Sobolev vector fields on ${\sf RCD}$ spaces, which are introduced via the theory of $L^0(\mm)$-modules, in fact can be defined up to ${\rm Cap}$-null sets and turn out to be continuous (in a sense to be made precise) outside sets of small capacities. Here the analogy is with the well-known case of Sobolev functions on the Euclidean space: these are a priori defined in a distributional-like sense, and thus up to equality $\mathcal L^d$-a.e., but once the concept of capacity is introduced one quickly realizes that a Sobolev function has a uniquely-defined representative up to ${\rm Cap}$-null sets which is continuous outside sets of small capacities.

\bigskip

More in detail, in this paper we do the following:
\begin{itemize}
\item[o)] We start recalling how to integrate w.r.t.\ an outer measure and that such integral is sublinear iff the outer measure is submodular, which is the case of capacity. This will allow to put a natural complete distance on the space $L^0(\Cap)$ of real-valued Borel functions on $\X$ identified up to $\Cap$-null sets. Given that $\Cap$-null sets are $\mm$-null, $L^0(\mm)$ can be seen as quotient of  $L^0(\Cap)$; we shall denote by $\Pr:L^0(\Cap)\to L^0(\mm)$ the quotient map.

We then recall the concept of quasi-continuous function which, being invariant under modification on $\Cap$-null sets, is a property of (equivalence classes of) functions in $L^0(\Cap)$. The space $\Cqc$ of quasi-continuous functions is actually a closed subspace of $L^0(\Cap)$ and coincides with the $L^0(\Cap)$-closure of continuous functions (then by approximation in the uniform norm one easily sees that in proper spaces one could also take the completion of the space of locally Lipschitz functions and, in $\R^d$, of smooth ones); we believe that this characterization of quasi-continuity is well-known in the literature but have not been able to find a reference -- in any case, for completeness in the preliminary section we provide full proofs of all the results we need. In connection with the concept of capacity the space $\Cqc$ is relevant for at least two reasons:
\begin{itemize}
\item[a)] The restriction of the projection operator $\Pr:L^0(\Cap)\to L^0(\mm)$ to $\Cqc$ is injective.
\item[b)] Any Sobolev function $f\in W^{1,2}(\X)\subset L^0(\mm)$ has a (necessarily unique, by a) above) quasi-continuous representative $\tilde f\in \Cqc$, i.e.\  $\Pr(\Cqc)\supset W^{1,2}(\X)$.
\end{itemize}

\item[i)] We propose the notion of $L^0({\rm Cap})$-normed $L^0({\rm Cap})$-modules ($L^0(\Cap)$-modules, in short), defined by properly imitating the one of $L^0(\mm)$-module. At the technical level an important difference between the two notions is that the capacity is only an outer measure: while in some cases this is only a nuisance (see e.g.\ the proof of the fact that the natural distance on $L^0(\Cap)$-modules satisfies the triangle inequality), in others it creates problems whose solution is unclear to us (e.g.\ in defining the dual of a $L^0({\rm Cap})$-module -- see Remark \ref{rmk:no_dual}).

We then see that, much like starting from $L^0(\Cap)$ and quotienting out up to $\mm$-null sets we find $L^0(\mm)$, starting from an arbitrary $L^0(\Cap)$-module $\mathscr M$ and quotienting via the relation
\begin{equation}\label{eq:eqrel}
v\sim w\quad\text{ provided }\Pr\big(|v-w|\big)=0\;\;\;\mm\text{-a.e.,}
\end{equation}
we produce a canonical $L^0(\mm)$-module $\mathscr M_\mm$ and  projection operator ${\rm Pr}_{\mathscr M}:\mathscr M \to \mathscr M_\mm$ (see Proposition \ref{pr:pr}).

\item[ii)] The main construction that we propose in this manuscript is that of tangent $L^0({\rm Cap})$-module $L^0_\Cap(T\X)$ on an ${\sf RCD}$ space $\X$. Specifically, in such setting we prove that there is a canonical couple $\big(L^0_\Cap(T\X),\bar\nabla\big)$, where $L^0_\Cap(T\X)$ is a  $L^0({\rm Cap})$-module and $\bar\nabla:\,{\rm Test}(\X)\to L^0_\Cap(T\X)$ is a linear map whose image generates $L^0_\Cap(T\X)$  and such that $|\bar\nabla f|$ coincides with the unique quasi-continuous representative of the minimal weak upper gradient $|Df|$ of $f$, see Theorem \ref{thm:tg_mod}. Here the space ${\rm Test}(\X)$ of test functions is made, in some sense, of the smoothest functions available on ${\sf RCD}$ spaces; this regularity matters in the definition of $L^0_\Cap(T\X)$ to the extent that  $|D f|$ belongs to $W^{1,2}(\X)$ whenever $f\in {\rm Test}(\X)$ (and this fact is in turn highly depending on the lower Ricci curvature bound: it seems hard to find many functions with this property on more general metric measure spaces).

The relation between $(L^0_\Cap(T\X),\bar\nabla)$ and the already known $L^0(\mm)$-tangent module $L^0_\mm(T\X)$ and gradient operator $\nabla$ is the fact that $L^0_\mm(T\X)$ can be seen as the quotient of $L^0_\Cap(T\X)$ via the equivalence relation \eqref{eq:eqrel}, where the projection operator  sends $\bar\nabla f$ to $\nabla f$ for any $f\in {\rm Test}(\X)$ (see Propositions \ref{pr:prtan}, \ref{pr:prtan2} for the precise formulation).

\item[iii)] We define the notion of `quasi-continuous vector field' in $L^0_\Cap(T\X)$. Here a relevant technical point is that there is no topology on the `tangent bundle' or, to put it differently, it is totally unclear what it means for a tangent vector field to be continuous or continuous at a point (in fact, not even the value of a vector field at a point is defined in our setting!). In this direction we also remark that the recent result in \cite{DPZ19} suggests that it might be pointless to look for `many' continuous vector fields already on finite dimensional Alexandrov spaces, thus a fortiori on ${\sf RCD}$ ones. 

Thus, much like in $\R^d$ quasi-continuous functions are the $L^0(\Cap)$-closure of smooth ones, we define the space of quasi-continuous vector fields $\Cqcvf$ as the $L^0_\Cap(T\X)$-closure of the space of the `smoothest' vector fields available, i.e.\ linear combinations of those of the form  $g\tilde\nabla f$ for $f,g\in {\rm Test}(\X)$. The choice of terminology is justified by the fact that the analogue of a), b) above hold (see Proposition \ref{prop:inj_bar_Pi} and Theorem \ref{thm:qcrvf}) and, moreover:
\begin{itemize}
\item[c)] for $v\in \Cqcvf$ we have $|v|\in\Cqc$ (see Proposition \ref{lem:|v|_qc}).
\end{itemize}
\end{itemize}

\bigskip

We conclude by pointing out that, while the concept of $L^0(\Cap)$-module makes sense for any $p$-capacity, for our purposes only the case $p=2$ is relevant. This is due to the fact that the natural Sobolev space to which $|Df|$ belongs for $f\in {\rm Test}(\X)$ is $W^{1,p}(\X)$ with $p=2$. Also, we remark that, albeit the definitions proposed in this paper are meant to be used in actual problems regarding the structure of ${\sf RCD}$ spaces (like the already mentioned one concerning the size of $\{\nabla f=0\}$ for $f$ harmonic - see Example \ref{ex:1d} for comments in this direction), in this manuscript we concentrate on building a solid foundation of the theoretical side of the story. The added value here is in providing what we believe are the correct definitions: once these are given, proofs of relevant results will come out rather easily.

\bigskip

\noindent{\bf Acknowledgments} This research has been supported by the MIUR SIR-grant `Nonsmooth Differential Geometry' (RBSI147UG4).

The authors want to thank Elia Bru\`e and Daniele Semola for stimulating conversations on the topics of this manuscript.

\section{Preliminaries}\label{s:preliminaries}
\subsection{Integration w.r.t.\ outer measures}\label{ss:integration_outer_measures}
Let $\X$ be a given set and $\mu$ an outer measure on $\X$.
Then for every function $f:\,\X\to[0,+\infty]$ (no measurability assumption is made here)
it holds that the function $[0,\infty)\ni t\mapsto\mu\big(\{f>t\}\big)\in[0,+\infty]$
is non-increasing and thus Lebesgue measurable. Hence the following
definition (via \emph{Cavalieri's formula})
\begin{equation}\label{eq:def_int}
\int f\,\d\mu:=\int_0^{+\infty}\mu\big(\{f>t\}\big)\,\d t.
\end{equation}
is well-posed. For an arbitrary set $E\subset\X$ we shall also put
\begin{equation}
\int_E f\,\d\mu:=\int\nchi_E\,f\,\d\mu.
\end{equation}

In the next result we collect the basic properties of the above-defined integral:
\begin{proposition}[Basic properties of $\mu$ and integrals w.r.t.\ it]
\label{prop:properties_mu}
The following properties hold:
\begin{itemize}
\item[a)] Let $f,g:\,\X\to[0,+\infty]$ be fixed. Then the following holds:
\begin{itemize}
\item[i)] $\int f\,\d\mu\leq\int g\,\d\mu$ provided $f\leq g$.
\item[ii)] $\int\lambda f\,\d\mu=\lambda\int f\,\d\mu$ for every $\lambda>0$.
\item[iii)] $\int f\,\d\mu=0$ if and only if $\mu\big(\{f\neq 0\}\big)=0$.
\item[iv)] $\int f\,\d\mu=\int g\,\d\mu$ provided $\mu\big(\{f\neq g\}\big)=0$.
\item[v)] \underline{\v{C}eby\v{s}\"{e}v's inequality.} It holds that
\[\mu\big(\{f\geq\lambda\}\big)\leq\frac{1}{\lambda}\int_{\{f\geq\lambda\}}f\,\d\mu
\quad\text{ for every }\lambda>0.\]
In particular, if $\int f\,\d\mu<+\infty$ then $\mu\big(\{f=+\infty\}\big)=0$.
\end{itemize}
\item[b)] \underline{Monotone convergence.} Let $f,f_n:\,\X\to[0,+\infty]$,
$n\in\N$, be such that
\[\mu\big(\{x\in\X\,:\,f_n(x)\not\,\uparrow f(x)\}\big)=0.\]
Then $\int f_n\,\d\mu\to\int f\,\d\mu$ as $n\to\infty$.
\item[c)] \underline{Borel-Cantelli.} Let $(E_n)_{n\in\N}$ be subsets of $\X$
satisfying $\sum_{n\in\N}\mu(E_n)<+\infty$.
Then it holds that $\mu\big(\bigcap_{n\in\N}\bigcup_{m\geq n}E_m\big)=0$.
\end{itemize}
\end{proposition}
\begin{proof}\ \\
\noindent{\bf (a)} $(i)$ is trivial and $(ii)$ follows by a change of variables.
The `if' implication in $(iii)$ is trivial, for the `only if' recall that
$t\mapsto\mu\big(\{f>t\}\big)$ is non-increasing to conclude that if $\int f\,\d\mu=0$
it must hold $\mu\big(\{f>t\}\big)=0$ for every $t>0$.
Then use the countable subadditivity of $\mu$ and the identity
$\{f>0\}=\bigcup_n\{f>1/n\}$ to deduce that $\mu\big(\{f>0\}\big)=0$.
To prove $(iv)$ notice that $\{g>t\}\subset\{f>t\}\cup\{f\neq g\}$,
hence taking into account the subadditivity of $\mu$ and the assumption
we get $\int g\,\d\mu\leq\int f\,\d\mu$. Inverting the roles of $f,g$ we conclude.
In order to prove $(v)$, call $E_\lambda:=\{f\geq\lambda\}$ and
notice that $E_\lambda\subset\{\nchi_{E_\lambda}f>t\}$ for all $t\in(0,\lambda)$, whence
\[\int_{E_\lambda}f\,\d\mu\geq\int_0^\lambda\mu\big(\{\nchi_{E_\lambda}f>t\}\big)\,\d t
\geq\lambda\,\mu(E_\lambda),\]
which proves $\mu(E_\lambda)\leq\lambda^{-1}\int_{E_\lambda}f\,\d\mu$.
For the last statement, notice that if $\int f\,\d\mu<+\infty$ then
\[\mu\big(\{f=+\infty\}\big)\leq\limi_{\lambda\to +\infty}\mu(E_\lambda)
\leq\lim_{\lambda\to +\infty}\frac{1}{\lambda}\int f\,\d\mu=0,\]
so that $\mu\big(\{f=+\infty\}\big)=0$, as required.

\noindent{\bf (b)} Assume for a moment that $f_n(x)\uparrow f(x)$ for every $x\in\X$.
Then the sequence of sets $\{f_n>t\}$ is increasing with respect to $n$ and satisfies
$\bigcup_n\{f_n>t\}=\{f>t\}$. Hence the monotone convergence theorem
(for the Lebesgue measure) gives
\[
\int f\,\d\mu=\int_0^{+\infty}\mu\big(\{f>t\}\big)\,\d t
=\lim_{n\to\infty}\int_0^{+\infty}\mu\big(\{f_n>t\}\big)\,\d t
=\lim_{n\to\infty}\int f_n\,\d\mu.
\]
The general case follows taking into account point $(iv)$.

\noindent{\bf (c)} The standard proof applies:
let us put $E:=\bigcap_{n\in\N}\bigcup_{m\geq n}E_m$, then
$\mu(E)\leq\mu\big(\bigcup_{m\geq n}E_m\big)$ for all $n\in\N$,
so that
\[\mu(E)\leq\limi_{n\to\infty}\mu\Big(\bigcup_{m\geq n}E_m\Big)
\leq\lim_{n\to\infty}\sum_{m\geq n}\mu(E_m)=0,\]
as required.
\end{proof}
\begin{example}\label{ex}{\rm
Consider the closed interval $S:=[0,1]$ in $\R$ (equipped with Euclidean
distance and Lebesgue measure). Given any $n\geq 1$, denote
by $P_n\subseteq S$ the singleton $\{1/n\}$. One can check that
$0<\Cap(P_n)<\Cap(S)$, but $\Cap(S\setminus P_n)=\Cap(S)$.
In other words, $\int\nchi_{S\setminus P_n}\,\d\Cap=\int\nchi_S\,\d\Cap$
and $\Cap\big(\{\nchi_{S\setminus P_n}\neq\nchi_S\}\big)>0$.
This shows that -- even for $\mu:=\Cap$ and $f\leq g$ --
the converse of item iv) of Proposition \ref{prop:properties_mu} fails.
\fr}\end{example}
\begin{remark}{\rm

An analogue of the dominated convergence theorem cannot hold,
as shown by the following counterexample. For any $n\geq 1$,
consider the singleton
$P_n$ in $\R$ defined in Example \ref{ex}.
Since the capacity in the space $\R$
is translation-invariant, one has that $\Cap(P_n)=\Cap(P_1)>0$ for all $n\geq 1$.
Moreover, we have $\lim_n\nchi_{P_n}(x)=0$ for all $x\in\R$
and $\nchi_{P_n}\leq\nchi_{[0,1]}$ for all $n\geq 1$,
with $\int\nchi_{[0,1]}\,\d\Cap=\Cap\big([0,1]\big)<+\infty$.
Nevertheless, it holds that $\int\nchi_{P_n}\,\d\Cap\equiv\Cap(P_1)$
does not converge to $0$ as $n\to\infty$, thus proving the failure
of the dominated convergence theorem. To provide such a counterexample,
we exploited the fact that the capacity is not $\sigma$-additive; indeed,
we built a sequence of pairwise disjoint sets, with the same positive capacity,
which are contained in a fixed set of finite capacity.
The lack of a result such as the dominated convergence theorem explains the
technical difficulties we will find in the proofs of Proposition
\ref{prop:Cauchy_indep_choice} and Theorem \ref{thm:L0_complete}.
\fr}\end{remark}
Let us now introduce a crucial property of outer measures:
\begin{definition}[Submodularity]
We say that $\mu$ is \emph{submodular} provided
\begin{equation}
\mu(E\cup F)+\mu(E\cap F)\leq\mu(E)+\mu(F)\quad\text{ for every }E,F\subset\X.
\end{equation}
\end{definition}
The importance of the above notion is due to the following result
(we refer to \cite[Chapter 6]{denneberg2010non} for a detailed bibliography):
\begin{theorem}[Subadditivity theorem]\label{thm:subadditivity_theorem}
It holds that $\mu$ is submodular if and only if the integral associated to $\mu$
is subadditive, i.e.\
\begin{equation}
\int(f+g)\,\d\mu\leq\int f\,\d\mu+\int g\,\d\mu
\quad\text{ for every }f,g:\,\X\to[0,+\infty).
\end{equation}
\end{theorem}
\begin{proof}
The `if' trivially follows by taking $f:=\nchi_E$ and $g:=\nchi_F$, so we turn to
the `only if'. Notice that, up to a monotone approximation argument based on 
point $(b)$ of Proposition \ref{prop:properties_mu}, it suffices to consider the case
in which $f,g$ assume only a finite number of values and
$\mu\big(\{f>0\}\cup\{g>0\}\big)<\infty$. Then up to replacing $\X$
with $\{f>0\}\cup\{g>0\}$ we can also assume that $\mu$ is finite.

Thus assume this is the case and let $\mathcal A$ be the (finite) algebra generated
by $f,g$, i.e.\ the one generated by the sets $\{f=a\}$ and $\{g=b\}$, $a,b\in\R$.
Let $A_1,\ldots,A_n\in\mathcal A$ be minimal with respect to inclusion among non-empty
sets in $\mathcal A$ and ordered in such a way that $a_i\geq a_{i+1}$ for every
$i=1,\ldots,n-1$, where $a_i:=(f+g)(A_i)$ (i.e.\ $a_i$ is the value of $f+g$ on $A_i$). 

Define a finite measure $\nu$ on $\mathcal A$ by putting 
\begin{equation}\label{eq:defnu}
\nu(A_1\cup\ldots\cup A_i):=\mu(A_1\cup\ldots\cup A_i)\quad\text{ for every }i=1,\ldots,n
\end{equation}
and notice that (since $\mu$ is finite) the measure $\nu$ is well-defined.
We claim that for a $\mathcal A$-measurable function $h:\,\X\to[0,+\infty]$ it holds
\begin{equation}\label{eq:claimnu}
\int h\,\d\mu\geq\int h\,\d\nu\quad
\text{ with equality if }h(A_i)\geq h(A_{i+1})\text{ for every }i=1,\ldots,n-1.
\end{equation}
Notice that once such claim is proved the conclusion easily follows from
$\int(f+g)\,\d\mu=\int(f+g)\,\d\nu$ (from the equality case of the claim \eqref{eq:claimnu}
and the choice of enumeration of the $A_i$'s), the inequalities
$\int f\,\d\nu\leq\int f\,\d\mu$, $\int g\,\d\nu\leq \int g\,\d\mu$ (from the claim
\eqref{eq:claimnu}) and the linearity of the integral w.r.t.\ $\nu$. 

Also, notice that the equality case in \eqref{eq:claimnu} is a direct consequence of
Cavalieri's formula for both $\mu,\nu$, the defining property \eqref{eq:defnu} and the
fact that for $h$ as stated it holds $\{h>t\}=A_1\cup\ldots\cup A_i$ for some $i=i(t)$.

Let us now assume that for some $\bar i$ it holds
$h(A_{\bar i})=:b_{\bar i}\leq b_{\bar i+1}:=h(A_{\bar i+1})$ and let us define
another finite measure $\tilde\nu$ as in \eqref{eq:defnu} with the sets $(A_i)_i$
replaced by $(\tilde A_i)_i$, where $\tilde A_i:= A_i$ for $i\neq \bar i,\bar i+1$,
$\tilde A_{\bar i}:=A_{\bar i+1}$ and $\tilde A_{\bar i+1}:=A_{\bar i}$. We shall prove that 
\begin{equation}\label{eq:toprove}
\int h\,\d\nu\leq \int h\,\d\tilde\nu
\end{equation}
and notice that this will give the proof, as with a finite number of such permutations
the sets $(A_i)_i$ get ordered as in the equality case in \eqref{eq:claimnu}.
Using Cavalieri's formula and noticing that by construction it holds
$\nu\big(\{h>t\}\big)=\tilde\nu\big(\{h>t\}\big)$ for $t\notin[b_{\bar i},b_{\bar i+1})$
we reduce to prove that
\[(b_{\bar i+1}-b_{\bar i})\,\nu\big(B\cup A_{\bar i+1}\big)=
\int_{b_{\bar i}}^{b_{\bar i+1}}\nu\big(\{h>t\}\big)\,\d t\leq
\int_{b_{\bar i}}^{b_{\bar i+1}}\tilde\nu\big(\{h>t\}\big)\,\d t=
(b_{\bar i+1}-b_{\bar i})\,\tilde\nu(B\cup A_{\bar i+1})\]
where $B:=A_1\cup\ldots\cup A_{\bar i-1}$. Given that $b_{\bar i+1}-b_{\bar i}\geq 0$,
the conclusion would follow if we showed that
$\nu(B\cup A_{\bar i+1})\leq\tilde\nu(B\cup A_{\bar i+1})$. Since
\[\begin{split}
\nu(B\cup A_{\bar i+1})&=\nu(B\cup A_{\bar i}\cup A_{\bar i+1})+\nu(B)-\nu(B\cup A_{\bar i})
\stackrel{\eqref{eq:defnu}}=\mu(B\cup A_{\bar i}\cup A_{\bar i+1})+
\mu(B)-\mu(B\cup A_{\bar i}),\\
\tilde\nu(B\cup A_{\bar i+1})&=\tilde \nu(B\cup\tilde A_{\bar i})
\stackrel{\eqref{eq:defnu}}=\mu(B\cup\tilde A_{\bar i})=\mu(B\cup A_{\bar i+1}),
\end{split}\]
we know from the submodularity of $\mu$ that
$\nu(B\cup A_{\bar i+1})\leq\tilde\nu(B\cup A_{\bar i+1})$, as required.
\end{proof}
\subsection{Capacity on metric measure spaces}\label{ss:capacity}
We shall be interested in a specific outer measure: the 2-capacity
(to which we shall simply refer as capacity) on a metric measure space $(\X,\sfd,\mm)$.

For the purposes of the current manuscript, by metric measure space we shall always
mean a triple $(\X,\sfd,\mm)$ such that
\begin{equation}\label{eq:mms}\begin{split}
(\X,\sfd)&\quad\text{ is a complete and separable metric space,}\\
\mm\geq 0&\quad\text{ is a Borel measure on }(\X,\sfd),\text{ finite on balls.}
\end{split}\end{equation}
In this setting, starting from \cite{Cheeger00}
(see also \cite{Shanmugalingam00,AmbrosioGigliSavare11}) there is a well-defined notion
of Sobolev space $W^{1,2}(\X,\sfd,\mm)$ (or, briefly, $W^{1,2}(\X)$) of real-valued
functions on $\X$, and to any $f\in W^{1,2}(\X,\sfd,\mm)$ is associated a function
$|D f|\in L^2(\mm)$, called minimal weak upper gradient, which plays the role of
the modulus of the distributional differential of $f$. For our purposes, it will be
useful to recall that $W^{1,2}(\X,\sfd,\mm)$ is a lattice (i.e.\ $f\vee g$ and
$f\wedge g$ are in $W^{1,2}(\X)$ provided $f,g\in W^{1,2}(\X)$),
that the minimal weak upper gradient is local, i.e.\
\begin{equation}
\label{eq:localwug}
|D f|=|Dg|\quad\mm\text{-a.e.\ on }\{f=g\},
\end{equation}
that if $f:\,\X\to\R$ is Lipschitz with bounded support, then
\begin{equation}
\label{eq:wuglip}
|D f|\leq{\rm lip}(f)\quad\mm\text{-a.e.,}
\end{equation}
where we set ${\rm lip}(f)(x):=\lims_{y\to x}\big|f(y)-f(x)\big|/\sfd(x,y)$ if $x$
is not isolated, $0$ otherwise.

Finally, the norm on $W^{1,2}(\X)$ is defined as
\[\|f\|_{W^{1,2}(\X)}^2:=\|f\|_{L^2(\mm)}^2+{\big\||Df|\big\|}^2_{L^2(\mm)}\]
and with this norm $W^{1,2}(\X)$ is always a Banach space whose norm is $L^2$-lower semicontinuous, i.e.
\begin{equation}
\label{eq:lscw12}
f_n\weakto f\ \text{ in }L^2(\mm)\quad\Longrightarrow\quad
\|f\|_{W^{1,2}(\X)}\leq\limi_{n\to\infty} \|f_n\|_{W^{1,2}(\X)},
\end{equation}
where as customary $\|f\|_{W^{1,2}(\X)}$ is set to be $+\infty$ if $f\notin W^{1,2}(\X)$.

Even if in general $W^{1,2}(\X)$ is not Hilbert (and thus the map
$f\mapsto\frac{1}{2}\int|Df|^2\,\d\mm$ is not a Dirichlet form),
the concept of capacity is well-defined as the definition carries over quite naturally (see e.g.\ \cite{Bjorn-Bjorn11}, \cite{HKST15} and references therein for the metric setting, and \cite{bouleau1991dirichlet} for the more classical framework of Dirichlet forms):
\begin{definition}[Capacity]
Let $E$ be a given subset of $\X$. Let us denote
\begin{equation}
\mathcal F_E:=\big\{f\in W^{1,2}(\X)\;\big|
\;f\geq 1\;\mm\text{-a.e.\ on some open neighbourhood of }E\big\}.
\end{equation}
Then the \emph{capacity} of the set $E$ is defined as
the quantity $\Cap(E)\in[0,+\infty]$, given by
\begin{equation}\label{eq:def_cap}
\Cap(E):=\inf_{f\in\mathcal F_E}{\|f\|}^2_{W^{1,2}(\X)}=
\inf_{f\in\mathcal F_E}\int |f|^2+|Df|^2\,\d\mm,
\end{equation}
with the convention that $\Cap(E):=+\infty$ whenever the family $\mathcal F_E$ is empty.
\end{definition}
In the following result we recall the main properties of the set-function $\Cap$:
\begin{proposition}\label{prop:properties_cap}
The capacity $\Cap$ is a submodular outer measure on $\X$,
which satisfies the following properties:
\begin{itemize}
\item[$\rm i)$] $\mm(E)\leq\Cap(E)$ for every Borel subset $E$ of $\X$,
\item[$\rm ii)$] $\Cap(B)<\infty$ for any bounded set $B\subset\X$.
\end{itemize}
\end{proposition}
\begin{proof}
We start claiming that for any $f,g\in W^{1,2}(\X)$ it holds
\begin{equation}\label{eq:capsub}
\|f\vee g\|_{W^{1,2}(\X)}^2+\|f\wedge g\|_{W^{1,2}(\X)}^2
=\|f\|_{W^{1,2}(\X)}^2+\|g\|_{W^{1,2}(\X)}^2.
\end{equation}
Indeed, the set $\big\{(f\vee g)(x),(f\wedge g)(x)\big\}$ coincides with the set
$\big\{f(x),g(x)\big\}$ for $\mm$-a.e.\ $x$ and thus
\[|f\vee g|^2(x)+|f\wedge g|^2(x)=|f|^2(x)+|g|^2(x)\]
and, similarly, by the locality property \eqref{eq:localwug} the set
$\big\{|D (f\vee g)|(x),|D(f\wedge g)|(x)\big\}$ coincides with the set
$\big\{|Df|(x),|Dg|(x)\big\}$ for $\mm$-a.e.\ $x$ and thus
\[\big|D(f\vee g)\big|^2(x)+\big|D(f\wedge g)\big|^2(x)=|Df|^2(x)+|Dg|^2(x).\]
Adding up these two identities and integrating we obtain \eqref{eq:capsub}.

Now let $E_1,E_2\subset\X$ be given and notice that for $f_i\in\mathcal F_{E_i}$, $i=1,2$,
it holds that $f_1\vee f_2\in\mathcal F_{E_1\cup E_2}$ and
$f_1\wedge f_2\in\mathcal F_{E_1\cap E_2}$, thus
\[\Cap(E_1\cup E_2)+\Cap(E_1\cap E_2)\leq
\|f_1\vee f_2\|_{W^{1,2}(\X)}^2+\|f_1\wedge f_2\|_{W^{1,2}(\X)}^2
\stackrel{\eqref{eq:capsub}}=\|f_1\|_{W^{1,2}(\X)}^2+\|f_2\|_{W^{1,2}(\X)}^2.\]
Hence passing to the infimum over $f_i\in\mathcal F_{E_i}$ we conclude that
\begin{equation}\label{eq:submodcap}
\Cap(E_1\cup E_2)+\Cap(E_1\cap E_2)\leq\Cap(E_1)+\Cap(E_2).
\end{equation}
In particular, this shows that $\Cap$ is finitely subadditive and thus to conclude
that it is a submodular outer measure it is sufficient to show that if $(E_n)_n$
is an increasing sequence of subsets of $\X$ it holds that
$\Cap\big(\bigcup_n E_n\big)=\sup_n\Cap(E_n)$. Since trivially
$\Cap(E)\leq\Cap(F)$ if $E\subset F$, it is sufficient to prove that
$\Cap\big(\bigcup_n E_n\big)\leq\sup_n\Cap(E_n)$ and this is obvious
if $\sup_n\Cap(E_n)=+\infty$. 

Thus assume that $S:=\sup_n\Cap(E_n)<+\infty$ and assume for the moment
also that the $E_n$'s are open. Let $f_n\in\mathcal F_{E_n}$ be such that
$\|f_n\|_{W^{1,2}(\X)}^2\leq\Cap(E_n)+2^{-n}\leq S+1$. Thus in particular
such sequence is bounded in $L^2(\mm)$ and thus for some $n_k\uparrow+\infty$
we have that $f_{n_k}\weakto f$ in $L^2(\mm)$ for some $f\in L^2(\mm)$.
Passing to the (weak) limit in $k$ in the inequality
$f_{n_k}\geq\nchi_{E_{n_k}}\geq \nchi_{E_{n_\ell}}$ valid for $k\geq\ell$,
we conclude that $f\geq \nchi_{E_{n_\ell}}$ for every $\ell$, hence
$f\geq\nchi_{\bigcup_n E_n}$. Since the $E_n$'s are open, this means that
$f\in\mathcal F_{\bigcup_n E_n}$. Therefore taking into account the semicontinuity property \eqref{eq:lscw12}
we deduce that
\[\Cap\big({\textstyle\bigcup\nolimits_n}E_n\big)\leq
\|f\|_{W^{1,2}(\X)}^2\leq\limi_{k\to\infty}\|f_{n_k}\|^2_{W^{1,2}(\X)}\leq
\limi_{k\to\infty}\Cap(E_{n_k})+\frac{1}{2^{n_k}}=\lim_{n\to\infty}\Cap(E_n).\]
Now let us drop the assumption that the $E_n$'s are open. Let $\eps>0$. We use the
submodularity property \eqref{eq:submodcap} and an induction argument to find an
increasing sequence $(\Omega_n)_n$ of open sets such that $\Omega_n\supset E_n$
and $\Cap(\Omega_n)\leq\Cap(E_n)+\eps\sum_{i\leq n}2^{-i}$. Then taking into account
what already proved for open sets we deduce that
\[\Cap\big({\textstyle\bigcup\nolimits_n}E_n\big)\leq
\Cap\big({\textstyle\bigcup\nolimits_n}\Omega_n\big)=
\lim_{n\to\infty}\Cap(\Omega_n)\leq\lim_{n\to\infty}\Cap(E_n)+
\eps\sum_{i\leq n}\frac{1}{2^i}=\lim_{n\to\infty}\Cap(E_n)+\eps\]
and the conclusion follows from the arbitrariness of $\eps>0$.

The inequality in $(i)$ trivially follows noticing that for $f\in\mathcal F_E $
it holds $f\geq 1$ $\mm$-a.e.\ on $E$, thus
\[\int|f|^2+|D f|^2\,\d\mm\geq\int_E 1\,\d\mm=\mm(E),\]
so that the conclusion follows taking the infimum over $f\in\mathcal F_E $.
For the statement $(ii)$ it is sufficient to recall that for any $B\subset\X$
bounded there is $f$ Lipschitz with bounded support that is $\geq 1$ on a 
neighbourhood of $B$ and that such $f$ belongs to $W^{1,2}(\X)$.
\end{proof}
\subsection{The space \texorpdfstring{$L^0(\Cap)$}{L0(Cap)}}
We have just seen that $\Cap$ is a submodular outer measure and in Subsection
\ref{ss:integration_outer_measures} we recalled how integration w.r.t.\ outer
measures is defined. It makes therefore sense to consider the integral associated
to $\Cap$ and that such integral is subadditive by Proposition \ref{prop:properties_cap}
and Theorem \ref{thm:subadditivity_theorem}.
In light of this observation, the following definition is meaningful:
\begin{definition}[The space $L^0(\Cap)$]\label{def:L0(Cap)}
Given any two functions $f,g:\,\X\to\R$, we will
say that $f=g$ \emph{in the $\Cap$-a.e.\ sense}
provided $\Cap\big(\{f\neq g\}\big)=0$. We define $L^0(\Cap)$ as the
space of all the equivalence classes -- up to $\Cap$-a.e.\ equality -- of
Borel functions on $\X$. 

We endow $L^0(\Cap)$ with the following distance: pick an increasing sequence
$(A_k)_k$ of open subsets of $\X$ with finite capacity such that for any
$B\subset\X$ bounded there is $k\in\N$ with $B\subset A_k$
(for instance, one could pick $A_k:=B_k(\bar x)$ for some $\bar x\in\X$),
then let us define
\begin{equation}\label{eq:def_distance_L0(Cap)}
\sfd_\Cap(f,g):=\sum_{k\in\N}\frac{1}{2^k\big(\Cap(A_k)\vee 1\big)}
\int_{A_k}|f-g|\wedge 1\,\d\Cap\quad\text{ for every }f,g\in L^0(\Cap).
\end{equation}
\end{definition}
Notice that the integral $\int_{A_k}|f-g|\wedge 1\,\d\Cap$ is well-defined,
since its value does not depend on the particular representatives of $f$ and $g$,
as granted by item $(iv)$ of Proposition \ref{prop:properties_mu}.
Moreover, we point out that the fact that $\sfd_\Cap$ satisfies the triangle
inequality is a consequence of the subadditivity of the integral associated
with the capacity.
\begin{remark}{\rm
We point out that if $\Cap(\X)<+\infty$ then the choice
$A_k:=\X$ for all $k\in\N$ is admissible in Definition \ref{def:L0(Cap)}.
\fr}\end{remark}
The next result shows that, even if the choice of the particular sequence $(A_k)_k$
might affect the distance $\sfd_\Cap$, its induced topology remains unaltered.
\begin{proposition}[Convergence in $L^0(\Cap)$]\label{prop:Cauchy_indep_choice}
The following holds:
\begin{itemize}
\item \underline{Characterization of Cauchy sequences}. 
Let $(f_n)_n\subseteq L^0(\Cap)$ be given. Then the following conditions are equivalent:
\begin{itemize}
\item[$\rm i)$] $\lim_{n,m}\sfd_\Cap(f_n,f_m)=0$,
\item[$\rm ii)$] $\lim_{n,m}\Cap\big(B\cap\big\{|f_n-f_m|>\eps\big\}\big)=0$
for any $\eps>0$ and any bounded set $B\subset\X$.
\end{itemize}
\item \underline{Characterization of convergence}.
Let $f\in L^0(\Cap)$ and $(f_n)_n\subseteq L^0(\Cap)$.
Then the following conditions are equivalent:
\begin{itemize}
\item[$\rm i)$] $\lim_n\sfd_\Cap(f_n,f)=0$,
\item[$\rm ii)$] $\lim_n\Cap\big(B\cap\big\{|f_n-f|>\eps\big\}\big)=0$ for
any $\eps>0$ and any bounded set $B\subset\X$.
\end{itemize}
\end{itemize}
\end{proposition}
\begin{proof}
We shall only prove the characterization of Cauchy sequences,
as the other claim follows by similar means.\\
{\color{blue}${\rm i)}\Longrightarrow{\rm ii)}$}
Fix any $0<\eps<1$ and a bounded set $B\subset\X$. Choose $k\in\N$ such that
$B\subset A_k$. Given that $\sfd_\Cap(f_n,f_m)\overset{n,m}\longrightarrow 0$,
we have $\lim_{n,m}\int_{A_k}|f_n-f_m|\wedge 1\,\d\Cap=0$.
Therefore we conclude that
\[\begin{split}
\lims_{n,m\to\infty}\eps\,\Cap\big(B\cap\big\{|f_n-f_m|>\eps\big\}\big)
&\leq\lims_{n,m\to\infty}\eps\,\Cap\big(A_k\cap\big\{|f_n-f_m|>\eps\big\}\big)\\
&\leq\lim_{n,m\to\infty}\int_{A_k}|f_n-f_m|\wedge 1\,\d\Cap=0.
\end{split}\]
{\color{blue}${\rm ii)}\Longrightarrow{\rm i)}$} Let $\eps>0$ be fixed. Choose $k\in\N$
such that $2^{-k}\leq\eps$. By our hypothesis, there is $\bar n\in\N$ such
that $\Cap\big(A_i\cap\big\{|f_n-f_m|>\eps\big\}\big)\leq\eps\,\Cap(A_i)$
for every $n,m\geq\bar n$ and $i=1,\ldots,k$. Let us call
$B_i^{nm}:=A_i\cap\big\{|f_n-f_m|>\eps\big\}$ and
$C_i^{nm}:=A_i\setminus B_i^{nm}$. Therefore for any $n,m\geq\bar n$ it holds that
\[\begin{split}
\sfd_\Cap(f_n,f_m)&\leq
\sum_{i=1}^k\frac{1}{2^i\,\Cap(A_i)}\int_{A_i}|f_n-f_m|\wedge 1\,\d\Cap
+\sum_{i=k+1}^\infty\frac{1}{2^i}\\
&\leq\sum_{i=1}^k\frac{1}{2^i\,\Cap(A_i)}\left[\int_{B_i^{nm}}|f_n-f_m|\wedge 1\,\d\Cap+
\int_{C_i^{nm}}|f_n-f_m|\wedge 1\,\d\Cap\right]+\eps\\
&\leq\sum_{i=1}^k\frac{1}{2^i\,\Cap(A_i)}\big[\Cap(B_i^{nm})+\eps\,\Cap(A_i)\big]
+\eps\leq 3\,\eps,
\end{split}\]
proving that $\lim_{n,m}\sfd_\Cap(f_n,f_m)=0$, as required.
\end{proof}

\begin{theorem}\label{thm:L0_complete}
The metric space $\big(L^0(\Cap),\sfd_\Cap\big)$ is complete.
\end{theorem}
\begin{proof}
Let $(f_n)_n$ be a $\sfd_\Cap$-Cauchy sequence of Borel functions $f_n:\,\X\to\R$.
Fix any $k\in\N$. Let $(f_{n_i})_i$ be an arbitrary subsequence of $(f_n)_n$.
Up to passing to a further (not relabeled) subsequence, it holds that
\begin{equation}\label{eq:L0_complete_aux1}
\Cap\big(A_k\cap\big\{|f_{n_i}-f_{n_{i+1}}|>2^{-i}\big\}\big)\leq 2^{-i}
\quad\text{ for every }i\in\N.
\end{equation}
Let us call $F_i:=A_k\cap\big\{|f_{n_i}-f_{n_{i+1}}|>2^{-i}\big\}$ for every $i\in\N$
and $F:=\bigcap_{i\in\N}\bigcup_{j\geq i}F_j$. Given that $\sum_{i\in\N}\Cap(F_i)<+\infty$ by
\eqref{eq:L0_complete_aux1}, we deduce from item c) of Proposition
\ref{prop:properties_mu} that $\Cap(F)=0$.
Notice that if $x\in A_k\setminus F=\bigcup_{i\in\N}\bigcap_{j\geq i}A_k\setminus F_j$,
then there is $i\in\N$ such that $\big|f_{n_j}(x)-f_{n_{j+1}}(x)\big|\leq 2^{-j}$
for all $j\geq i$, which grants that $\big(f_{n_i}(x)\big)_i\subseteq\R$ is a Cauchy
sequence for every $x\in A_k\setminus F$.
Therefore we define the Borel function $g^k:\,A_k\to\R$ as
\[g^k(x):=\left\{\begin{array}{ll}
\lim_i f_{n_i}(x)\\
0
\end{array}\quad\begin{array}{ll}
\text{ if }x\in A_k\setminus F,\\
\text{ if }x\in F.
\end{array}\right.\]
Now fix any $\eps>0$. Choose $\bar i\in\N$ such that $\sum_{i\geq\bar i}2^{-i}\leq\eps$.
If $i\geq\bar i$ and $x\in\bigcap_{j\geq i}A_k\setminus F_j$
(thus in particular $x \notin F$), hence one has
$\big|f_{n_i}(x)-g^k(x)\big|\leq\sum_{j\geq i}\big|f_{n_j}(x)-f_{n_{j+1}}(x)\big|\leq\sum_{j\geq i}2^{-j}\leq\eps$.
This implies that
\begin{equation}
A_k \cap\big\{|f_{n_i}-g^k|>\eps\big\}\subseteq\bigcup_{j\geq i}F_j
\quad\text{ for every }i\geq\bar i.
\end{equation}
Then $\Cap\big(A_k\cap\big\{|f_{n_i}-g^k|>\eps\big\}\big)
\leq\sum_{j\geq i}\Cap(F_j)\leq\sum_{j\geq i}2^{-j}$
holds for every $i\geq\bar i$, thus
\begin{equation}
\lim_{i\to\infty}\Cap\big(A_k\cap\big\{|f_{n_i}-g^k|>\eps\big\}\big)=0
\quad\text{ for every }\eps>0.
\end{equation}
We proved this property for some subsequence of a given
subsequence $(f_{n_i})_i$ of $(f_n)_n$, hence this shows that
\begin{equation}\label{eq:L0_complete_aux2}
\lim_{n\to\infty}\Cap\big(A_k\cap\big\{|f_n-g^k|>\eps\big\}\big)=0
\quad\text{ for every }\eps>0.
\end{equation}
Now let us define the Borel function $f:\,\X\to\R$ as $f:=\sum_{k\in\N}\nchi_{A_k}\,g^k$.
Notice that the identity $A_k\cap\big\{|f_n-f|>\eps\big\}=A_k\cap\big\{|f_n-g^k|>\eps\big\}$
and \eqref{eq:L0_complete_aux2} yield
\[\lim_{n\to\infty}\Cap\big(A_k\cap\big\{|f_n-f|>\eps\big\}\big)=0
\quad\text{ for every }k\in\N\text{ and }\eps>0.\]
Since any bounded subset of $\X$ is contained in the set $A_k$ for some $k\in\N$,
we immediately deduce that $\lim_n\Cap\big(B\cap\big\{|f_n-f|>\eps\big\}\big)=0$
whenever $\eps>0$ and $B\subset\X$ is bounded. This grants that $\lim_n\sfd_\Cap(f_n,f)=0$
by Proposition \ref{prop:Cauchy_indep_choice}, thus proving that
$\big(L^0(\Cap),\sfd_\Cap\big)$ is a complete metric space and accordingly the statement.
\end{proof}

We conclude this section with some other basic properties
of the metric space $\big(L^0(\Cap),\sfd_\Cap\big)$:
\begin{proposition}\label{prop:Sf_dense}
Let $f_n\to f$
in $L^0(\Cap)$. Then there exists a subsequence $n_j\uparrow+\infty$ such that
for $\Cap$-a.e.\ $x$ it holds that $f_{n_j}(x)\to f(x)$.

Moreover, the space ${\sf Sf}(\X)$ of simple functions, which is defined as
\begin{equation}
{\sf Sf}(\X):=\left\{\sum_{n=1}^\infty\alpha_n\,\nchi_{E_n}\;\bigg|\;
\begin{array}{ll}
(\alpha_n)_n\subset\R\text{ and }(E_n)_n\text{ is some}\\
\text{partition of }\X\text{ into Borel sets}\end{array}\right\},
\end{equation}
is dense in $\big(L^0(\Cap),\sfd_\Cap\big)$.
\end{proposition}
\begin{proof}
As for the standard case of measures, let the subsequence satisfy
$\sfd_\Cap(f_{n_j},f_{n_{j+1}})\leq 2^{-j}$ for all $j\in\N$. By the very definition
of $\sfd_\Cap$, we deduce that for every $j,k\in\N$ one has
\begin{equation}\label{eq:Sf_dense_aux}
\int_{A_k}|f_{n_j}-f_{n_{j+1}}|\wedge 1\,\d\Cap\leq\frac{c_k}{2^j},
\quad\text{ where }c_k:=2^k\big(\Cap(A_k)\vee 1\big).
\end{equation}
Calling $g_j(x):=\sum_{i=1}^{j-1}\big|f_{n_i}(x)-f_{n_{i+1}}(x)\big|\wedge 1$
for every $j\in\N$ and $x\in\X$, we see that $g_j(x)\uparrow g_\infty(x)$ for
some $g_\infty:\,\X\to[0,+\infty]$. Given any $k\in\N$, we know from item b) of
Proposition \ref{prop:properties_mu} (and the subadditivity of the integral
associated to $\Cap$) that
\[\int_{A_k}g_\infty\,\d\Cap=\lim_{j\to\infty}\int_{A_k}g_j\,\d\Cap
\leq\limi_{j\to\infty}\sum_{i=1}^{j-1}\int_{A_k}|f_{n_i}-f_{n_{i+1}}|\wedge 1\,\d\Cap
\overset{\eqref{eq:Sf_dense_aux}}\leq c_k\limi_{j\to\infty}\sum_{i=1}^{j-1}\frac{1}{2^i}
=c_k.\]
Therefore item v) of Proposition \ref{prop:properties_mu} ensures that
$g_\infty(x)<+\infty$ for $\Cap$-a.e.\ $x\in A_k$, whence also for $\Cap$-a.e.\ $x\in\X$
by arbitrariness of $k\in\N$. Now observe that for all $j'\geq j$ and
$\Cap$-a.e.\ $x\in\X$ it holds that
\begin{equation}\label{eq:Sf_dense_aux2}
\big|f_{n_{j'}}(x)-f_{n_j}(x)\big|\wedge 1\leq
\sum_{i=j}^{j'-1}\big|f_{n_i}(x)-f_{n_{i+1}}(x)\big|\wedge 1
=g_{j'}(x)-g_j(x)\leq g_\infty(x)-g_j(x)<+\infty.
\end{equation}
By letting $j,j'\to\infty$ in \eqref{eq:Sf_dense_aux2} we deduce that
$\big(f_{n_j}(x)\big)_j\subset\R$ is Cauchy for $\Cap$-a.e.\ $x\in\X$,
thus it admits a limit $\tilde f(x)\in\R$. Again by item b) of Proposition
\ref{prop:properties_mu} we know for every $k\in\N$ that
\begin{equation}\label{eq:Sf_dense_aux3}\begin{split}
\int_{A_k}(g_\infty-g_j)\,\d\Cap&\overset{\phantom{\eqref{eq:Sf_dense_aux}}}=
\lim_{j'\to\infty}\int_{A_k}\sum_{i=j}^{j'}|f_{n_i}-f_{n_{i+1}}|\wedge 1\,\d\Cap\\
&\overset{\phantom{\eqref{eq:Sf_dense_aux}}}\leq
\limi_{j'\to\infty}\sum_{i=j}^{j'}\int_{A_k}|f_{n_i}-f_{n_{i+1}}|\wedge 1\,\d\Cap\\
&\overset{\eqref{eq:Sf_dense_aux}}\leq c_k\lim_{j'\to\infty}\sum_{i=j}^{j'}\frac{1}{2^i}
=\frac{c_k}{2^{j-1}}\overset{j}\longrightarrow 0.
\end{split}\end{equation}
By letting $j'\to\infty$ in \eqref{eq:Sf_dense_aux2} we get
$|\tilde f-f_{n_j}|\wedge 1\leq g_\infty-g_j$ $\Cap$-a.e., whence
for any $k\in\N$ it holds
\[\begin{split}
\int_{A_k}|\tilde f-f|\wedge 1\d\Cap&\leq
\limi_{j\to\infty}\bigg[\int_{A_k}|\tilde f-f_{n_j}|\wedge 1\,\d\Cap
+\int_{A_k}|f_{n_j}-f|\wedge 1\,\d\Cap\bigg]\\
&\leq\limi_{j\to\infty}\bigg[\int_{A_k}(g_\infty-g_j)\,\d\Cap
+c_k\,\sfd_\Cap(f_{n_j},f)\bigg]\overset{\eqref{eq:Sf_dense_aux3}}=0.
\end{split}\]
This means that $\int_{A_k}|\tilde f-f|\wedge 1\,\d\Cap=0$ for every $k\in\N$,
thus accordingly $\tilde f=f$ holds $\Cap$-a.e.\ by item iii) of
Proposition \ref{prop:properties_mu}. We then finally conclude that
$f_{n_j}(x)\to f(x)$ for $\Cap$-a.e.\ $x\in\X$.

For the second statement we argue as follows. Fix $f\in L^0(\Cap)$ and $\eps>0$.
Choose a Borel representative $\bar f:\,\X\to\R$ of $f$.
For any integer $i\in\Z$, let us define $E_i:=\bar f^{-1}\big([i\,\eps,(i+1)\,\eps)\big)$.
Then $(E_i)_{i\in\Z}$ constitutes a partition of $\X$ into Borel sets, so that
$\bar g:=\sum_{i\in\Z}i\,\eps\,\nchi_{E_i}$ is a well-defined Borel function that belongs
to ${\sf Sf}(\X)$. Finally, it holds that $\big|\bar f(x)-\bar g(x)\big|<\eps$ for
every $x\in\X$, which grants that $\sfd_\Cap(f,g)\leq\eps$, where
$g\in L^0(\Cap)$ denotes the equivalence class of $\bar g$. Hence the statement follows.
\end{proof}
\begin{remark}\label{rmk:d_conv_implies_ae-notinverse}{\rm
In general, $\Cap$-a.e.\ convergence does not imply convergence in $L^0(\Cap)$,
as shown by the following counterexample.
Consider $P_n$ as in Example \ref{ex} for any $n\geq 1$.
We have that the functions $f_n:=\nchi_{P_n}$ pointwise converge to $0$ as $n\to\infty$.
However, it holds that
\[\Cap\big([0,1]\cap\big\{|f_n|>1/2\big\}\big)=\Cap(P_n)\equiv\Cap(P_1)>0\]
does not converge to $0$, thus we do not have $\lim_n\sfd_\Cap(f_n,0)=0$
by Proposition \ref{prop:Cauchy_indep_choice}.
\fr}\end{remark}
\subsection{Quasi-continuous functions and quasi-uniform convergence}
Here we quickly recall the definition and main properties of quasi-continuous functions associated to Sobolev functions (see \cite{Bjorn-Bjorn11}, \cite{Koskela-Rajala-Shanmugalingam03},  \cite{bouleau1991dirichlet} for more on the topic and detailed bibliography).

\begin{definition}[Quasi-continuous functions]
We say that a function $f:\,\X\to\R$ is \emph{quasi-continuous} provided
for every $\eps>0$ there exists a set $E\subset\X$ with
$\Cap(E)<\eps$ such that the function
$f\restr{\X\setminus E}:\,\X\setminus E\to\R$ is continuous.
\end{definition}
It is clear that if $f,\tilde f$ agree $\Cap$-a.e.\ and one of them is quasi-continuous,
so is the other. Also, by the very definition of capacity, in defining quasi-continuity
one could restrict to sets $E$ which are open. In particular, if $f$ is quasi-continuous
there is an increasing sequence $(C_n)_n$ of closed subsets of $\X$ with
$\lim_n\Cap(\X\setminus C_n)=0$ such that $f$ is continuous on each $C_n$.
Then $N:=\bigcap_n\X\setminus C_n$ is a Borel set with null capacity --
in particular, we have $\mm(N)=0$ by item i) of Proposition \ref{prop:properties_cap}
-- and $f$ is Borel on $\X\setminus N$. This proves that any quasi-continuous function
is $\mm$-measurable and $\Cap$-a.e.\ equivalent to a Borel function. 

We shall denote by $\Cqc$ the collection of all equivalence classes --
up to $\Cap$-a.e.\ equality -- of quasi-continuous functions on $\X$.
What we just said ensures that $\Cqc\subset L^0(\Cap)$.
It is readily verified that $\Cqc$ is an algebra.

Let us now discuss a notion of convergence particularly relevant in relation with $\Cqc$:
\begin{definition}[Local quasi-uniform convergence]
Let $f_n:\,\X\to\R$, $n\in\N\cup\{\infty\}$ be Borel functions.
Then we say that $f_n$ \emph{locally quasi-uniformly converges} to $f_\infty$
as $n\to\infty$ provided for any $B\subset\X$ bounded
and any $\eps>0$ there exists a set $E\subset\X$ with $\Cap(E)<\eps$
such that $f_n\to f_\infty$ uniformly on $B\setminus E$.
In this case, we shall write $f_n\stackrel\qu\to f_\infty$.
\end{definition}
As before, nothing changes if one even requires the set $E$ to be open in the above
definition and the notion of local quasi-uniform convergence is invariant under
modification of the functions in $\Cap$-null sets. Local quasi-uniform convergence
is (almost) the convergence induced by the following distance:
\begin{equation}\label{eq:defdqu}
\sfd_\qu(f,g):=\inf_{E\subset\X}\sum_{k\in\N}\bigg(\frac{\Cap(E\cap A_k)}
{2^k\big(\Cap(A_k)\vee 1\big)}+\frac{1}{2^k}
\sup_{x\in A_k\setminus E}\big|f(x)-g(x)\big|\wedge 1\bigg),
\end{equation}
where $(A_k)_k$ is any sequence as in Definition \ref{def:L0(Cap)}.
Indeed, it is trivial to verify that $\sfd_\qu$ is actually a distance
(notice that $\sfd_\qu\leq 1$, as one can see by picking $E=\X$ in \eqref{eq:defdqu}),
moreover we have:
\begin{proposition}\label{prop:dqu}
Let  $f_n:\X\to \R$, $n\in\N\cup\{\infty\}$
be Borel functions. Then
\begin{itemize}
\item[i)] If $f_n\stackrel\qu\to f_\infty$, then $\sfd_\qu(f_n,f_\infty)\to 0$.
\item[ii)]  If $\sfd_\qu(f_n,f_\infty)\to 0$, then any subsequence $n_k$ has a further
subsequence, not relabeled, such that $f_{n_k}\stackrel\qu\to f_\infty$.
\end{itemize}
\end{proposition}
\begin{proof}\ \\
{\bf (i)} Let $\eps>0$ and use the definition of local quasi-uniform convergence
to find some subsets $(E_k)_{k\in\N}$ of $\X$ such that $\Cap(E_k)<\eps/2^k$ and
$f_n\to f_\infty$ uniformly on $A_k\setminus E_k$ for any $k\in\N$.
Choosing the set $E:=\bigcup_{k\in\N}E_k$ in \eqref{eq:defdqu} yields
(for $\bar k\in\N$ sufficiently big)
\[\begin{split}\lims_{n\to\infty}\sfd_\qu(f_n,f_\infty)&\leq
\eps+\lims_{n\to\infty}\sum_{k=1}^{\bar k}
\bigg(\frac{\Cap(E\cap A_k)}
{2^k\big(\Cap(A_k)\vee 1\big)}+\frac{1}{2^k}
\sup_{x\in A_k\setminus E}\big|f_n(x)-f_\infty(x)\big|\wedge 1\bigg)\\
&\leq\eps+\lims_{n\to\infty}\sum_{k=1}^{\bar k}\bigg(\frac{\eps}
{2^k}+\sup_{x\in A_k\setminus E_k}\big|f_n(x)-f_\infty(x)\big|\bigg)\\
&\leq 2\,\eps+\sum_{k=1}^{\bar k}\lims_{n\to\infty}
\sup_{x\in A_k\setminus E_k}\big|f_n(x)-f_\infty(x)\big|=2\,\eps,
\end{split}\]
and the conclusion follows by the arbitrariness of $\eps$.\\
{\bf (ii)} We shall prove that if $\sum_n\sfd_\qu(f_n,f_\infty)<+\infty$
then $f_n\stackrel\qu\to f_\infty$. First, choose a sequence $(E_n)_n$
of subsets of $\X$ such that
\begin{equation}\label{eq:dqu_aux}
\sum_{n\in\N}\sum_{k\in\N}\bigg(\frac{\Cap(E_n\cap A_k)}
{2^k\big(\Cap(A_k)\vee 1\big)}+\frac{1}{2^k}
\sup_{A_k\setminus E_n}|f_n-f_\infty|\wedge 1\bigg)<+\infty.
\end{equation}
Let $\eps>0$ and $B\subset\X$ bounded be fixed. Pick $\bar k\in\N$ such
that $B\subset A_{\bar k}$. Then \eqref{eq:dqu_aux} grants the existence
of $\bar n\in\N$ with $\sum_{n\geq\bar n}\Cap(E_n\cap A_{\bar k})<\eps$,
thus $E:=\bigcup_{n\geq\bar n}E_n\cap A_{\bar k}$ satisfies $\Cap(E)<\eps$.
Therefore we have that
\[\sum_{n\geq\bar n}\sup_{B\setminus E}|f_n-f_\infty|\wedge 1
\leq\sum_{n\geq\bar n}\sup_{A_{\bar k}\setminus E_n}|f_n-f_\infty|\wedge 1
=2^{\bar k}\sum_{n\geq\bar n}\frac{1}{2^{\bar k}}
\sup_{A_{\bar k}\setminus E_n}|f_n-f_\infty|\wedge 1
\overset{\eqref{eq:dqu_aux}}<+\infty,\]
whence accordingly $f_n\to f_\infty$ uniformly on $B\setminus E$.
This grants that $f_n\stackrel\qu\to f_\infty$, as required.
\end{proof}
\begin{proposition}\label{prop:QC_complete}
The following properties hold:
\begin{itemize}
\item[i)] The metric space $\big(\Cqc,\sfd_\qu\big)$ is complete.
\item[ii)] It holds that
$\sfd_\Cap(f,g)\leq\sfd_\qu(f,g)\leq 2\,\sqrt{\sfd_\Cap(f,g)}$ for every
$f,g\in\Cqc$. In particular, the canonical embedding of $\Cqc$ in $L^0(\Cap)$
is continuous and has closed image.
\item[iii)] $\Cqc$ is the closure in $L^0(\Cap)$ of the space of (equivalence classes up to $\Cap$-null sets of) continuous functions. \end{itemize}
\end{proposition}
\begin{proof}\ \\
{\bf (i)} To prove completeness, fix a $\sfd_\qu$-Cauchy sequence $(f_n)_n$ of
quasi-continuous functions. Up to passing to a (not relabeled) subsequence,
we can suppose that $\sfd_\qu(f_n,f_{n+1})<2^{-n}$ for all $n$. For any $n\in\N$
we can pick a set $E_n\subset\X$ such that the function $f_n$ is continuous
on $\X\setminus E_n$ and
\begin{equation}\label{eq:QC_complete_aux}
\sum_{k\in\N}\bigg(\frac{\Cap(E_n\cap A_k)}{2^k\big(\Cap(A_k)\vee 1\big)}
+\frac{1}{2^k}\sup_{A_k\setminus E_n}|f_n-f_{n+1}|\wedge 1\bigg)<\frac{1}{2^n}.
\end{equation}
Now define $F_{k,m}:=\bigcup_{n\geq m}E_n\cap A_k$ and
$F_k:=\bigcap_{m\in\N}F_{k,m}$ for every $k,m\in\N$.
Hence one has
\[\sum_{n\geq m}\sup_{A_k\setminus F_{k,m}}|f_n-f_{n+1}|\wedge 1
\leq\sum_{n\geq m}\sup_{A_k\setminus E_n}|f_n-f_{n+1}|\wedge 1
\overset{\eqref{eq:QC_complete_aux}}<2^k\sum_{n\geq m}\frac{1}{2^n}<+\infty\]
for any given $k,m\in\N$, so that $f_n\stackrel n\to g_m$ uniformly on the set
$A_k\setminus F_{k,m}$ for some continuous function $g_{k,m}:\,A_k\setminus F_{k,m}\to\R$.
Let us set
\[f_\infty(x):=\left\{\begin{array}{ll}
g_{k,m}(x)\\
0
\end{array}\quad\begin{array}{ll}
\text{ if }x\in A_k\setminus F_{k,m}\text{ for some }k,m\in\N,\\
\text{ if }x\in\X\setminus\big(\bigcup_{k\in\N}A_k\setminus F_k\big).
\end{array}\right.\]
Clearly $f_\infty$ is well-defined as $g_{k,m}=g_{k',m'}$ on
$(A_k\setminus F_{k,m})\cap(A_{k'}\setminus F_{k',m'})$ for all $k,k',m,m'\in\N$.
Moreover, we know from \eqref{eq:QC_complete_aux} that
$\sum_{n\in\N}\Cap(E_n\cap A_k)\leq
2^k\big(\Cap(A_k)\vee 1\big)\sum_{n\in\N}2^{-n}<+\infty$ holds for every $k\in\N$,
whence item c) of Proposition \ref{prop:properties_mu} ensures that
$\Cap(F_k)=0$ for all $k\in\N$. By item b) of Proposition \ref{prop:properties_mu}
we see that
\[\begin{split}
\Cap\big(\X\setminus\big({\textstyle\bigcup\nolimits_{k\in\N}}A_k\setminus F_k\big)\big)
&=\lim_{j\to\infty}\Cap\big(A_j\setminus
\big({\textstyle\bigcup\nolimits_{k\in\N}}A_k\setminus F_k\big)\big)
\leq\limi_{j\to\infty}\Cap\big(A_j\setminus(A_j\setminus F_j)\big)\\
&=\lim_{j\to\infty}\Cap(F_j)=0,
\end{split}\]
which shows that the function $f_\infty$ is quasi-continuous.
We claim that $f_n\stackrel\qu\to f_\infty$, which is enough to conclude by item i)
of Proposition \ref{prop:dqu}. Given any $\eps>0$ and any bounded set
$B\subset\X$, we can pick $\bar k,\bar n\in\N$ such that $B\subset A_{\bar k}$
and $\Cap(E)<\eps$, where we set $E:=E_{\bar n}\cap A_{\bar k}$.
Therefore we have
\[\sup_{B\setminus E}|f_n-f_\infty|\leq
\sup_{A_{\bar k}\setminus E_{\bar n}}|f_n-g_{\bar k,\bar n}|\overset{n}\longrightarrow 0,\]
which implies that $f_n\stackrel\qu\to f_\infty$, as desired.\\
%
{\bf (ii)} Fix $f,g\in\Cqc$ and take $(A_k)_k$ as in Definition \ref{def:L0(Cap)}.
Given any $E\subset\X$, it holds that
\[\begin{split}
\int_{A_k}|f-g|\wedge 1\,\d\Cap&\leq
\int_{E\cap A_k}|f-g|\wedge 1\,\d\Cap+\int_{A_k\setminus E}|f-g|\wedge 1\,\d\Cap\\
&\leq\Cap(E\cap A_k)+\Cap(A_k)\sup_{A_k\setminus E}|f-g|\wedge 1,
\end{split}\]
whence accordingly
\[\frac{1}{2^k\big(\Cap(A_k)\vee 1\big)}\int_{A_k}|f-g|\wedge 1\,\d\Cap
\leq\frac{\Cap(E\cap A_k)}{2^k\big(\Cap(A_k)\vee 1\big)}
+\frac{1}{2^k}\sup_{A_k\setminus E}|f-g|\wedge 1\]
for every $k\in\N$. By summing over $k\in\N$ and then passing to the infimum
over $E\subset\X$, we conclude that $\sfd_\Cap(f,g)\leq\sfd_\qu(f,g)$.

On the other hand, let us consider the set $E_\lambda:=\big\{|f-g|\wedge 1>\lambda\big\}$
for any $\lambda>0$. Therefore for every $k\in\N$ one has that
$\Cap(E_\lambda\cap A_k)\leq\lambda^{-1}\int_{A_k}|f-g|\wedge 1\,\d\Cap$
by item v) of Proposition \ref{prop:properties_mu} and that
$\sup_{A_k\setminus E_\lambda}|f-g|\wedge 1\leq\lambda$, thus accordingly
\[\begin{split}
\sfd_\qu(f,g)&\leq\sum_{k\in\N}\bigg(\frac{\Cap(E_\lambda\cap A_k)}
{2^k\big(\Cap(A_k)\vee 1\big)}+\frac{1}{2^k}
\sup_{x\in A_k\setminus E_\lambda}\big|f(x)-g(x)\big|\wedge 1\bigg)\\
&\leq\lambda+\frac{1}{\lambda}\sum_{k\in\N}\frac{1}{2^k\big(\Cap(A_k)\vee 1\big)}
\int_{A_k}|f-g|\wedge 1\,\d\Cap
=\lambda+\frac{\sfd_\Cap(f,g)}{\lambda}.
\end{split}\]
By letting $\lambda\downarrow\sqrt{\sfd_\Cap(f,g)}$ we conclude
that $\sfd_\qu(f,g)\leq 2\,\sqrt{\sfd_\Cap(f,g)}$, as required.\\
{\bf (iii)} Let $f:\X\to\R$ be a Borel function whose equivalence class up to $\Cap$-null sets belongs to $\Cqc$ and $\eps>0$. Then by definition there is an open set $\Omega$ with $\Cap(\Omega)<\eps$ and $f\restr{\X\setminus\Omega}$ is continuous. By the Tietze extension theorem there is $g\in C(\X)$ which agrees with $f$ on $\X\setminus\Omega$, and -- since this latter condition ensures that $\sfd_\qu(f,g)<\eps$ -- the proof is achieved.
\end{proof}

We now turn to the relation between quasi-continuity and Sobolev functions,
and to do so it is useful to emphasise whether we speak about functions up
to $\Cap$-null sets or up to $\mm$-null sets. We shall therefore write
$[f]_\Cap$ (resp.\ $[f]_\mm$) for the equivalence class of the Borel function
$f:\,\X\to\R$ up to $\Cap$-null (resp.\ $\mm$-null) sets. 

We start noticing that -- since $\mm$ is absolutely continuous with respect
to $\Cap$, i.e.\ $\Cap$-null sets are also $\mm$-null (recall $(i)$ of Proposition \ref{prop:properties_cap}) -- there is a natural projection map
\begin{equation}\label{eq:Pi_fcs}\begin{split}
\Pr:\,L^0(\Cap)&\longrightarrow L^0(\mm),\\
[f]_\Cap&\longmapsto[f]_\mm.
\end{split}\end{equation}
Since in general there are $\mm$-null sets which are not $\Cap$-null, such projection
operator is typically non-injective. This is why the following result is interesting:
\begin{proposition}[Uniqueness of quasi-continuous representative] 
\label{prop:eqmmimplieseqCap}
Let $f,g:\,\X\to\R$ be quasi-continuous functions. Then $f=g$ $\mm$-a.e.\ implies $f=g$
$\Cap$-a.e.. In other words,
\[\Pr\restr{\Cqc}:\,\Cqc\longrightarrow L^0(\mm)\]
is an injective map.
\end{proposition}
\begin{proof}
Let $N:=\{f\neq g\}$. Let $\Omega\subset\X$ open be such that $f,g$
are continuous on $\X\setminus\Omega$. Thus $N$ is open in $\X\setminus\Omega$
and therefore $\tilde \Omega:=N\cup\Omega$ is open in $\X$. By assumption we
know that $\mm(N)=0$ and thus the very definition of capacity yields
$\Cap(\Omega)=\Cap(\tilde\Omega)$. Hence
\[\Cap(N)\leq \Cap(\tilde\Omega)=\Cap(\Omega)\]
and the quasi-continuity assumption gives the conclusion.
\end{proof}
\begin{proposition}\label{prop:linkqusob}
Let $f,g\in C(\X)$ be such that $[f]_\mm,[g]_\mm\in W^{1,2}(\X)$. Then
\[\sfd_\qu\big([f]_\Cap,[g]_\Cap\big)\leq
3\,{\big\|[f]_\mm-[g]_\mm\big\|}_{W^{1,2}(\X)}^{\frac{2}{3}}.\] 
\end{proposition}
\begin{proof}
For any $\lambda>0$ let $\Omega_\lambda:=\big\{|f-g|>\lambda\big\}$,
so that by definition of $\sfd_\qu$ we have
\begin{equation}\label{eq:linkqusob_aux}
\sfd_\qu\big([f]_\Cap,[g]_\Cap\big)\leq\lambda+
\sum_{k\in\N}\frac{\Cap(\Omega_\lambda\cap A_k)}{2^k\big(\Cap(A_k)\vee 1\big)}.
\end{equation}
Notice that $\Omega_\lambda$ is an open set by continuity of $|f-g|$.
Moreover, $\lambda^{-1}\,\big[|f-g|\big]_\mm$ is a Sobolev function satisfying
$\lambda^{-1}\,\big[|f-g|\big]_\mm\geq 1$ $\mm$-a.e.\ on $\Omega_\lambda$.
Hence $\Cap(\Omega_\lambda\cap A_k)\leq\lambda^{-2}
\,{\big\|[f-g]_\mm\big\|}_{W^{1,2}(\X)}^2$ holds for all $k\in\N$.
Plugging this estimate in \eqref{eq:linkqusob_aux} we obtain that
\[\sfd_\qu\big([f]_\Cap,[g]_\Cap\big)\leq\lambda+2
\frac{{\big\|[f]_\mm-[g]_\mm\big\|}_{W^{1,2}(\X)}^2}{\lambda^2},\]
then by choosing $\lambda:={\big\|[f]_\mm-[g]_\mm\big\|}_{W^{1,2}(\X)}^{\frac{2}{3}}$
we get the conclusion.
\end{proof}
Collecting these last two propositions we obtain the following result:
\begin{theorem}[Quasi-continuous representative of Sobolev function] 
\label{thm:qcrepresentative}
Suppose that (equivalence classes up to $\mm$-a.e.\ equality of)
continuous functions in $W^{1,2}(\X)$ are dense in $W^{1,2}(\X)$.
Then there exists a unique continuous map
\begin{equation}
\qcr:\,W^{1,2}(\X)\longrightarrow\Cqc
\end{equation}
such that the composition $\Pr\circ \qcr$
is the inclusion map $W^{1,2}(\X)\subset L^0(\mm)$.

Moreover, $\qcr$ is linear and satisfies
\begin{equation}
\label{eq:normqcr}
\big|\qcr (f)\big|=\qcr\big(|f|\big)\;\;\;\Cap\text{-a.e.}
\quad\text{ for every }f\in W^{1,2}(\X).
\end{equation}
Finally, if $[f_n]_\mm\to [f]_\mm$ in $W^{1,2}(\X)$, then any subsequence
has a further subsequence converging locally quasi-uniformly. 
\end{theorem}
\begin{proof}
For $f\in C(\X)$ with $[f]_\mm\in W^{1,2}(\X)$ the requirements for $\qcr(f)$
are that it must belong to $\Cqc$ and satisfy $\Pr\big(\qcr(f)\big)=[f]_\mm$.
Thus Proposition \ref{prop:eqmmimplieseqCap} forces it to be equal to $[f]_\Cap$.
Proposition \ref{prop:linkqusob} ensures that such assignment is Lipschitz as map
from $W^{1,2}(\X)\cap C(\X)$ to $\Cqc$, and thus can be uniquely extended to a continuous
map on the whole $W^{1,2}(\X)$. 

Since $\qcr$ is linear on continuous functions, by continuity it is linear
on the whole $W^{1,2}(\X)$. \eqref{eq:normqcr} is trivial for continuous functions,
thus its validity for general ones follows by continuity.

The last statement is a direct consequence of what already proved
and Proposition \ref{prop:dqu}.
\end{proof}
\section{Main result}
\subsection{\texorpdfstring{$L^0(\Cap)$}{L0(Cap)}-normed \texorpdfstring{$L^0(\Cap)$}{L0(Cap)}-modules}
The language of $L^0(\mm)$-normed $L^0(\mm)$-modules over a metric
measure space $(\X,\sfd,\mm)$ has been proposed and investigated by
the second author in \cite{Gigli14}, with the final aim of developing a
differential calculus on $\sf RCD$ spaces. In the present paper, we
assume the reader to be familiar with such language. We shall use the
term \emph{$L^0(\mm)$-module} in place of $L^0(\mm)$-normed
$L^0(\mm)$-module and we will typically denote by $\mathscr M_\mm$ any
such object. We refer to \cite{Gigli14,Gigli17} for a detailed account
about this topic. Here we introduce a new notion of normed module, called 
\emph{$L^0(\Cap)$-normed $L^0(\Cap)$-module} or, more simply, \emph{$L^0(\Cap)$-module}, in which the measure under consideration
is the capacity $\Cap$ instead of the reference measure $\mm$.
\medskip

Let $(\X,\sfd,\mm)$ be a metric measure space as in \eqref{eq:mms}
and $(A_k)_k$ a sequence as in Definition \ref{def:L0(Cap)}.
\begin{definition}[$L^0(\Cap)$-normed $L^0(\Cap)$-module]\label{def:n_Cap-m}
We say that a quadruple $\big(\mathscr M,\tau,\,\cdot\,,|\cdot|\big)$
is a \emph{$L^0(\Cap)$-normed $L^0(\Cap)$-module} over $(\X,\sfd,\mm)$ provided:
\begin{itemize}
\item[$\rm i)$] $(\mathscr M,\tau)$ is a topological vector space.
\item[$\rm ii)$] The bilinear map $\,\cdot:\,L^0(\Cap)\times\mathscr M
\to\mathscr M$ satisfies $f\cdot(g\cdot v)=(fg)\cdot v$ and $1\cdot v=v$
for every $f,g\in L^0(\Cap)$ and $v\in\mathscr M$.
\item[$\rm iii)$] The map $|\cdot|:\,\mathscr M\to L^0(\Cap)$, called
\emph{pointwise norm}, satisfies
\[\begin{split}
|v|\geq 0&\quad\text{ for every }v\in\mathscr M,\text{ with equality
if and only if }v=0,\\
|v+w|\leq|v|+|w|&\quad\text{ for every }v,w\in\mathscr M,\\
|f\cdot v|=|f||v|&\quad\text{ for every }v\in\mathscr M\text{ and }
f\in L^0(\Cap),
\end{split}\]
where all equalities and inequalities are intended in the $\Cap$-a.e.\ sense.
\item[$\rm iv)$] The distance $\sfd_{\mathscr M}$ on $\mathscr M$, given by
\[\sfd_{\mathscr M}(v,w):=
\sum_{k\in\N}\frac{1}{2^k\big(\Cap(A_k)\vee 1\big)}\int_{A_k}|v-w|\wedge 1\,\d\Cap
\quad\text{ for all }v,w\in\mathscr M,\]
is complete and induces the topology $\tau$.
\end{itemize}
\end{definition}
%
%
%
%
%
\medskip

Much like starting from $L^0(\Cap)$ and passing to the quotient up to $\mm$-a.e.\ equality we obtain $L^0(\mm)$, in the same way by passing to an appropriate quotient starting from an arbitrary $L^0(\Cap)$-module we obtain a $L^0(\mm)$-module. Let us describe this procedure.

Let $\mathscr M$ be a $L^0(\Cap)$-module and define an equivalence relation on it by declaring
\[
v\sim_\mm w\quad\Longleftrightarrow\quad|v-w|=0\;\;\mm\text{-a.e.\ in }\X.
\]
Then we consider the quotient $\mathscr M_\mm:=\mathscr M/\sim_\mm$, the projection map $\Pr_{\mathscr M}:\,\mathscr M\to \mathscr M_\mm$ sending $v$ to its equivalence class $[v]_\mm$ and define the following operations on $\mathscr M_\mm$:
\[\begin{split}
[v]_\mm+[w]_\mm&:=[v+w]_\mm\in\mathscr M_\mm,\\
[f]_\mm\cdot[v]_\mm&:=\big[[f]_\Cap\cdot v\big]_\mm\in\mathscr M_\mm,\\
\big|[v]_\mm\big|&:=\Pr\big(|v|\big)\in L^0(\mm),
\end{split}\]
for every $v,w\in\mathscr M$ and $f:\,\X\to\R$ Borel, where
$\Pr:\,L^0(\Cap)\to L^0(\mm)$ is the projection operator defined in \eqref{eq:Pi_fcs}.  Routine verifications show that with these operations $\mathscr M_\mm$ is a $L^0(\mm)$-module.

For a given $L^0(\Cap)$-module $\mathscr M$, the couple $(\mathscr M_\mm,{\rm Pr}_{\mathscr M})$ is characterized by the following universal property:
\begin{proposition}[Universal property of $(\mathscr M_\mm,{\rm Pr}_{\mathscr M})$]\label{pr:pr} Let $\mathscr M$ be a $L^0(\Cap)$-module and let $(\mathscr M_\mm,{\rm Pr}_{\mathscr M})$ be defined as above. Also, let $\mathscr N_\mm$ be a $L^0(\mm)$-module and $T:\mathscr M\to\mathscr N_\mm$ be  a linear map satisfying 
\begin{equation}
\label{eq:ipT}
\big|T(v)\big|\leq\Pr\big(|v|\big)\;\;\;\mm\text{-a.e.}\quad\text{ for every }v\in\mathscr M.
\end{equation}
Then there is a unique $L^0(\mm)$-linear and continuous map $T_{\rm Pr}:\mathscr M_\mm\to\mathscr N_\mm$ such that the diagram
\begin{equation}
\label{eq:diag}
\begin{tikzpicture}[node distance=2cm, auto]
  \node (S) {$\mathscr M$};
  \node (C) [right of=S] {$\mathscr M_\mm$};
  \node (L) [below  of=C] {$\mathscr N_\mm$};
  \draw[->] (S) to node {${\rm Pr}_{\mathscr M}$} (C);
  \draw[->] (C) to node {$T_{\rm Pr}$} (L);
  \draw[->] (S) to node [swap] {$T$} (L);
\end{tikzpicture}
\end{equation}
commutes. 

In particular, for any other couple $(\mathscr M'_\mm,{\rm Pr}'_{\mathscr M})$ with the same property there is a unique isomorphism $\Phi:\mathscr M_\mm\to \mathscr M'_\mm$ (i.e.\ bijection which preserves the whole structure of $L^0(\Cap)$-module) such that $\Phi\circ\Pr_{\mathscr M}=\Pr_{\mathscr M}'$.
\end{proposition}
\begin{proof} The latter statement is an obvious consequence of the former, so we concentrate on this one. Let $v,w\in \mathscr M$ be such that $v\sim_\mm w$ and notice that
\[
\big|T(v)-T(w)\big|=\big|T(v-w)\big|\stackrel{\eqref{eq:ipT}}\leq{\rm Pr}\big(|v-w|\big)=0
\quad\text{ holds }\mm\text{-a.e.}.
\]
Thus $T$ passes to the quotient and defines a map $T_{\rm Pr}:\mathscr M_\mm\to \mathscr N_\mm$ making the diagram \eqref{eq:diag} commute. It is clear that $T_{\rm Pr}$ is linear and continuous (the latter being a consequence of \eqref{eq:ipT} and the definition), thus to conclude it is sufficient to prove $L^0(\mm)$-linearity. By linearity and continuity this will follow if we show that $T_{\rm Pr}\big([\nchi_E]_\mm[v]_\mm\big)=[\nchi_E]_\mm\,T_{\rm Pr}\big([v]_\mm\big)$ for any Borel set $E\subset \X$; in turn, this will follow if we prove that
\[
T\big([\nchi_E]_\Cap\,v\big)=[\nchi_E]_\mm\,T(v)\quad\text{ for every }v\in\mathscr M
\text{ and }E\subset\X\text{ Borel.}
\] 
To show this, notice that from \eqref{eq:ipT} it follows $\big|T([\nchi_{E^c}]_\Cap\,v)\big|\leq [\nchi_{E^c}]_\mm\,{\rm Pr}\big(|v|\big)$, thus multiplying both sides by $[\nchi_{E}]_\mm$ we obtain
\begin{equation}
\label{eq:car}
[\nchi_{E}]_\mm\,T\big([\nchi_{E^c}]_\Cap\,v\big)=0\quad\text{and symmetrically}\quad [\nchi_{E^c}]_\mm\,T\big([\nchi_{E}]_\Cap\,v\big)=0.
\end{equation}
Therefore
\[
\begin{split}
T\big([\nchi_{E}]_\Cap\,v\big)&\stackrel{\phantom{\eqref{eq:car}}}=\big([\nchi_{E}]_\mm +[\nchi_{E^c}]_\mm\big)\,T\big([\nchi_{E}]_\Cap\,v\big)\\
&\stackrel{\eqref{eq:car}}=[\nchi_{E}]_\mm\,T\big([\nchi_{E}]_\Cap\,v\big)\\
&\stackrel{\eqref{eq:car}}=[\nchi_{E}]_\mm \big(T([\nchi_{E}]_\Cap\,v)+T([\nchi_{E^c}]_\Cap\,v)\big)=[\nchi_{E}]_\mm\,T(v)
\end{split}
\]
and the conclusion follows.
\end{proof}

\begin{remark}\label{rmk:no_dual}{\rm
In analogy with the case of $L^0(\mm)$-modules, one could be tempted
to define the dual of a  $L^0(\Cap)$-module $\mathscr M_\Cap$ as the
space of $L^0(\Cap)$-linear continuous maps $L:\,\mathscr M_\Cap\to L^0(\Cap)$
and to declare that the pointwise norm $|L|$ of any such $L$ is the minimal
element of $L^0(\Cap)$ (where minimality is intended in the $\Cap$-a.e.\ sense)
such that the inequality $|L|\geq L(v)$ holds $\Cap$-a.e.\ for any $v\in\mathscr M_\Cap$
that $\Cap$-a.e.\ satisfies $|v|\leq 1$.

Technically speaking, for $L^0(\mm)$-modules this can be achieved by using
the notion of \emph{essential supremum} of a family of Borel functions.
Nevertheless, it seems that this tool cannot be adapted to the situation in which
we want to consider the capacity instead of the reference measure,
as suggested by Example \ref{ex}. 
\fr}\end{remark}
\begin{definition}\label{def:Hilbert_module}
Let $\mathscr H$ be a $L^0(\Cap)$-module over $(\X,\sfd,\mm)$. Then
we say that $\mathscr H$ is a \emph{Hilbert module} provided
\begin{equation}\label{eq:Hilbert_module}
|v+w|^2+|v-w|^2=2\,|v|^2+2\,|w|^2\quad\text{ holds }\Cap\text{-a.e.\ in }\X
\end{equation}
for every $v,w\in\mathscr H$.
\end{definition}
\medskip

By polarisation, we define a pointwise scalar product
$\la\cdot,\cdot\ra:\,\mathscr H\times\mathscr H\to L^0(\Cap)$ as
\begin{equation}
\la v,w\ra:=\frac{|v+w|^2-|v|^2-|w|^2}{2}\quad\Cap\text{-a.e.\ in }\X.
\end{equation}
Then the operator $\la\cdot,\cdot\ra$ is $L^0(\Cap)$-bilinear and satisfies
\begin{equation}\begin{split}
\big|\la v,w\ra\big|&\leq|v||w|\\
\la v,v\ra&=|v|^2
\end{split}\qquad\Cap\text{-a.e.}\text{ for every }v,w\in\mathscr H.
\end{equation}
\subsection{Tangent \texorpdfstring{$L^0(\Cap)$}{L0(Cap)}-module}
Let $(\X,\sfd,\mm)$ be an ${\sf RCD}(K,\infty)$ space, for some $K\in\R$.
A fundamental class of Sobolev functions on $\X$ is that of \emph{test functions},
denoted by ${\rm Test}(\X)$ (cf.\ \cite{Gigli14}).
We point out that we are in a position to apply Theorem \ref{thm:qcrepresentative} above,
since Lipschitz functions with bounded support are dense in $W^{1,2}(\X)$,
as proven in \cite{AmbrosioGigliSavare11-3}. Moreover, a fact that is fundamental
for our discussion (see \cite{Savare13}) is the following:
\begin{equation}\label{eq:prod_test_Sob}
\la\nabla f,\nabla g\ra\in W^{1,2}(\X)\quad\text{ for every }f,g\in{\rm Test}(\X).
\end{equation}
In particular, by taking $g=f$ in \eqref{eq:prod_test_Sob} we get
$|Df|^2\in W^{1,2}(\X)$ for every $f\in{\rm Test}(\X)$.
\medskip

Let us use the notation $L^0_\mm(T\X)$ to indicate the tangent $L^0(\mm)$-module
over $(\X,\sfd,\mm)$. Recall that ${\rm TestV}(\X)\subseteq L^0_\mm(T\X)$ denotes
the class of \emph{test vector fields} on $\X$, while $H^{1,2}_C(T\X)$ is the closure
of ${\rm TestV}(\X)$ in the Sobolev space $W^{1,2}_C(T\X)$.
We know from \cite[Proposition 2.19]{Gigli17} that for any
$v\in H^{1,2}_C(T\X)\cap L^\infty_\mm(T\X)$ one has that $|v|^2\in W^{1,2}(\X)$ and
\begin{equation}\label{eq:comp_metric}
\d|v|^2(w)=2\,\la\nabla_w v,v\ra\;\;\;\mm\text{-a.e.}
\quad\text{ for every }w\in L^0_\mm(T\X),
\end{equation}
whence in particular $\big|D|v|^2\big|\leq 2\,|\nabla v|_{\sf HS}\,|v|$ holds $\mm$-a.e..
This in turn implies the following:
\begin{lemma}\label{lem:|v|_Sobolev}
Let $(\X,\sfd,\mm)$ be an ${\sf RCD}(K,\infty)$ space, for some $K\in\R$.
Let $v\in H^{1,2}_C(T\X)$ be fixed. Then $|v|\in W^{1,2}(\X)$ and
\begin{equation}\label{eq:ineq_grad_|v|}
\big|D|v|\big|\leq|\nabla v|_{\sf HS}\quad\text{ holds }\mm\text{-a.e.\ in }\X.
\end{equation}
\end{lemma}
\begin{proof}
First of all, we prove the statement for $v\in{\rm TestV}(\X)$. Given any $\eps>0$,
let us define the Lipschitz function $\varphi_\eps:\,[0,+\infty)\to\R$ as
$\varphi_\eps(t):=\sqrt{t+\eps}$ for any $t\geq 0$. Hence by applying the chain rule
for minimal weak upper gradients we see that $\varphi_\eps\circ|v|^2\in{\rm S}^2(\X)$
(cf.\ \cite{AmbrosioGigliSavare11} for the notion of Sobolev class ${\rm S}^2(\X)$) and
\[\big|D(\varphi_\eps\circ|v|^2)\big|=\varphi'_\eps\circ|v|^2\,\big|D|v|^2\big|
=\frac{\big|D|v|^2\big|}{2\,\sqrt{|v|^2+\eps}}\leq\frac{|v|}{\sqrt{|v|^2+\eps}}
\,|\nabla v|_{\sf HS}\leq|\nabla v|_{\sf HS}.\]
This grants the existence of $G\in L^2(\mm)$ and a sequence $\eps_j\searrow 0$
such that $\big|D(\varphi_{\eps_j}\circ|v|^2)\big|\weakto G$
weakly in $L^2(\mm)$ as $j\to\infty$ and $G\leq|\nabla v|_{\sf HS}$ in the
$\mm$-a.e.\ sense. Since $\varphi_{\eps_j}\circ|v|^2\to|v|$ pointwise $\mm$-a.e.\ as
$j\to\infty$, we deduce from the lower semicontinuity of minimal weak upper gradients
that $|v|\in W^{1,2}(\X)$ and that $\big|D|v|\big|\leq|\nabla v|_{\sf HS}$
holds $\mm$-a.e.\ in $\X$.

Now fix $v\in H^{1,2}_C(T\X)$. Pick a sequence
$(v_n)_n\subseteq{\rm TestV}(\X)$ that $W^{1,2}_C(T\X)$-converges to $v$.
In particular, $|v_n|\to|v|$ and $|\nabla v_n|_{\sf HS}\to|\nabla v|_{\sf HS}$ in
$L^2(\mm)$. By the first part of the proof we know that $|v_n|\in W^{1,2}(\X)$
and $\big|D|v_n|\big|\leq|\nabla v_n|_{\sf HS}$ for all $n\in\N$, thus accordingly
(up to a not relabeled subsequence) we have that $\big|D|v_n|\big|\weakto H$ weakly
in $L^2(\mm)$, for some $H\in L^2(\mm)$ such that $H\leq|\nabla v|_{\sf HS}$ holds
$\mm$-a.e.\ in $\X$. Again by lower semicontinuity of minimal weak upper gradients
we conclude that $|v|\in W^{1,2}(\X)$ with $\big|D|v|\big|\leq|\nabla v|_{\sf HS}$
in the $\mm$-a.e.\ sense, proving the statement.
\end{proof}
\medskip

We now introduce the so-called \emph{tangent $L^0(\Cap)$-module} $L^0_\Cap(T\X)$ over $\X$,
which is a  $L^0(\Cap)$-module in the sense of Definition \ref{def:n_Cap-m}.
\begin{theorem}[Tangent $L^0(\Cap)$-module]\label{thm:tg_mod}
Let $(\X,\sfd,\mm)$ be an ${\sf RCD}(K,\infty)$ space.
Then there exists a unique couple $\big(L^0_\Cap(T\X),\bar\nabla\big)$,
where $L^0_\Cap(T\X)$ is a  $L^0(\Cap)$-module over $\X$ and the operator
$\bar\nabla:\,{\rm Test}(\X)\to L^0_\Cap(T\X)$ is linear, such that
the following properties hold:
\begin{itemize}
\item[$\rm i)$] For any $f\in{\rm Test}(\X)$ we have that the equality
$|\bar\nabla f|=\qcr\big(|Df|\big)$ holds $\Cap$-a.e.\ on $\X$
(note that $|Df|\in W^{1,2}(\X)$ as a consequence of Lemma \ref{lem:|v|_Sobolev}).
\item[$\rm ii)$] The space of $\sum_{n\in\N}\nchi_{E_n}\bar\nabla f_n$, with
$(f_n)_n\subseteq{\rm Test}(\X)$ and $(E_n)_n$ Borel partition of $\X$,
is dense in $L^0_\Cap(T\X)$.
\end{itemize}
Uniqueness is intended up to unique isomorphism: given another couple
$(\mathscr M_\Cap,\bar\nabla')$ with the same properties, there exists a unique
isomorphism $\Phi:\,L^0_\Cap(T\X)\to\mathscr M_\Cap$ with $\Phi\circ\bar\nabla=\bar\nabla'$.

The space $L^0_\Cap(T\X)$ is called \emph{tangent $L^0(\Cap)$-module} associated
to $(\X,\sfd,\mm)$, while its elements are said to be \emph{$\Cap$-vector fields} on $\X$.
Moreover, the operator $\bar\nabla$ is called \emph{gradient}.
\end{theorem}
\begin{proof}\ \\
{\color{blue}\textsc{Uniqueness.}} Consider any simple vector field
$v\in L^0_\Cap(T\X)$, i.e.\ $v=\sum_{n\in\N}\nchi_{E_n}\bar\nabla f_n$
for some $(f_n)_n\subseteq{\rm Test}(\X)$ and $(E_n)_n$ Borel partition of $\X$.
We are thus forced to set
\begin{equation}\label{eq:def_isom_cotg_mod}
\Phi(v):=\sum_{n\in\N}\nchi_{E_n}\bar\nabla'f_n\in\mathscr M_\Cap.
\end{equation}
Such definition is well-posed, as granted by the $\Cap$-a.e.\ equalities
\[\left|\sum_{n\in\N}\nchi_{E_n}\bar\nabla'f_n\right|
=\sum_{n\in\N}\nchi_{E_n}|\bar\nabla'f_n|
=\sum_{n\in\N}\nchi_{E_n}|Df_n|
=\sum_{n\in\N}\nchi_{E_n}|\bar\nabla f_n|=|v|,\]
which also show that $\Phi$ preserves the pointwise norm of simple vector fields.
In particular, the map $\Phi$ is linear and continuous, whence it can be
uniquely extended to a linear and continuous operator
$\Phi:\,L^0_\Cap(T\X)\to\mathscr M_\Cap$ by density of simple vector fields
in $L^0_\Cap(T\X)$. It follows from Proposition \ref{prop:Sf_dense}
that $\Phi$ preserves the pointwise norm. Moreover, we know from the definition
\eqref{eq:def_isom_cotg_mod} that $\Phi(fv)=f\,\Phi(v)$ is
satisfied for any simple $f$ and $v$, whence also for all $f\in L^0(\Cap)$
and $v\in L^0_\Cap(T\X)$ by Proposition \ref{prop:Sf_dense}.
To conclude, just notice that the image of $\Phi$ is dense in $\mathscr M_\Cap$
by density of simple vector fields in $\mathscr M_\Cap$, thus accordingly
$\Phi$ is surjective (as its image is closed, being $\Phi$ an isometry).
Therefore we proved that there exists a unique module isomorphism
$\Phi:\,L^0_\Cap(T\X)\to\mathscr M_\Cap$ such that $\Phi\circ\bar\nabla=\bar\nabla'$,
as required.\\
{\color{blue}\textsc{Existence.}} We define the `pre-tangent module' $\sf Ptm$
as the set of all sequences $(E_n,f_n)_n$, where $(f_n)_n\subseteq{\rm Test}(\X)$
and $(E_n)_n$ is a Borel partition of $\X$. We now define an equivalence relation
$\sim$ on $\sf Ptm$: we declare that $(E_n,f_n)_n\sim(F_m,g_m)_m$ provided
\[\qcr\big(|D(f_n-g_m)|\big)=0\;\;\;\Cap\text{-a.e.\ on }E_n\cap F_m
\quad\text{ for every }n,m\in\N.\]
The equivalence class of $(E_n,f_n)_n$ will be denoted by $[E_n,f_n]_n$.
Moreover, let us define
\[\alpha\,[E_n,f_n]_n+\beta\,[F_m,g_m]_m:=[E_n\cap F_m,\alpha\,f_n+\beta\,g_m]_{n,m}\]
for every $\alpha,\beta\in\R$ and $[E_n,f_n]_n,[F_m,g_m]_m\in{\sf Ptm}/\sim$,
so that ${\sf Ptm}/\sim$ inherits a vector space structure; well-posedness of
these operations is granted by the locality property of minimal weak upper
gradients and by Theorem \ref{thm:qcrepresentative}.
We define the pointwise norm of any given element $[E_n,f_n]_n\in{\sf Ptm}/\sim$ as
\begin{equation}\label{eq:def_ptwse_norm}
\big|[E_n,f_n]_n\big|:=\sum_{n\in\N}\nchi_{E_n}\qcr\big(|Df_n|\big)\in L^0(\Cap).
\end{equation}
Then we define $L^0_\Cap(T\X)$ as the completion of the metric space
$\big({\sf Ptm}/\sim\,,\sfd_{L^0_\Cap(T\X)}\big)$, where
\begin{equation}\label{eq:def_distance_cotg_mod}
\sfd_{L^0_\Cap(T\X)}(v,w):=
\sum_{k\in\N}\frac{1}{2^k\big(\Cap(A_k)\vee 1\big)}\int_{A_k}|v-w|\wedge 1\,\d\Cap
\quad\text{ for all }v,w\in{\sf Ptm}/\sim,
\end{equation}
while we set $\bar\nabla f:=[\X,f]\in L^0_\Cap(T\X)$ for every test function
$f\in{\rm Test}(\X)$, thus obtaining a linear operator
$\bar\nabla:\,{\rm Test}(\X)\to L^0_\Cap(T\X)$. Item i) of the statement is thus
clearly satisfied. Observe that $[E_n,f_n]_n=\sum_{n\in\N}\nchi_{E_n}\bar\nabla f_n$
for every $[E_n,f_n]_n\in{\sf Ptm}/\sim$, so that also item ii) is verified,
as a consequence of the density of ${\sf Ptm}/\sim$ in $L^0_\Cap(T\X)$.
Now let us define the multiplication operator
$\,\cdot:\,{\sf Sf}(\X)\times({\sf Ptm}/\sim)\to{\sf Ptm}/\sim$ as follows:
\begin{equation}\label{eq:def_mult}
\bigg(\sum_{m\in\N}\alpha_m\,\nchi_{F_m}\bigg)\cdot[E_n,f_n]_n:=
[E_n\cap F_m,\alpha_m\,f_n]_{n,m}\in{\sf Ptm}/\sim.
\end{equation}
Therefore the maps defined in \eqref{eq:def_ptwse_norm} and
\eqref{eq:def_mult} can be uniquely extended by continuity to
a pointwise norm operator $|\cdot|:\,L^0_\Cap(T\X)\to L^0(\Cap)$ and
a multiplication by $L^0(\Cap)$-functions
$\,\cdot:\,L^0(\Cap)\times L^0_\Cap(T\X)\to L^0_\Cap(T\X)$,
respectively. It also turns out that the distance $\sfd_{L^0_\Cap(T\X)}$ is
expressed by the formula in \eqref{eq:def_distance_cotg_mod} for any
$v,w\in L^0_\Cap(T\X)$, as one can readily deduce from Proposition
\ref{prop:Sf_dense}. Finally, standard verifications show that
$L^0_\Cap(T\X)$ is a $L^0(\Cap)$-module over $(\X,\sfd,\mm)$,
thus concluding the proof.
\end{proof}
\begin{remark}{\rm
An analogous construction has been carried out in \cite{Gigli14} to define
the cotangent $L^0(\mm)$-module $L^0_\mm(T^*\X)$, while the tangent $L^0(\mm)$-module
$L^0_\mm(T\X)$ was obtained from the cotangent one by duality.
However, since we cannot consider duals of $L^0(\Cap)$-modules
(as pointed out in Remark \ref{rmk:no_dual}), we opted for a different
axiomatisation. We just underline the fact that, since $\sf RCD$ spaces are
infinitesimally Hilbertian, the modules $L^0_\mm(T^*\X)$ and $L^0_\mm(T\X)$
can be canonically identified via the Riesz isomorphism.
\fr}\end{remark}
\begin{proposition}
The tangent $L^0(\Cap)$-module $L^0_\Cap(T\X)$ is a Hilbert module.
\end{proposition}
\begin{proof}
Given any $f,g\in{\rm Test}(\X)$, we deduce from item i) of Theorem \ref{thm:tg_mod}
and the last statement of Theorem \ref{thm:qcrepresentative} that
\[\begin{split}
|\bar\nabla f+\bar\nabla g|^2+|\bar\nabla f-\bar\nabla g|^2
&=\qcr\big(|D(f+g)|^2+|D(f-g)|^2\big)=\qcr\big(2\,|Df|^2+2\,|Dg|^2\big)\\
&=2\,|\bar\nabla f|^2+2\,|\bar\nabla g|^2.
\end{split}\]
This grants that the pointwise parallelogram identity \eqref{eq:Hilbert_module}
is satisfied whenever $v,w$ are $L^0(\Cap)$-linear combinations of elements
of $\big\{\bar\nabla f\,:\,f\in{\rm Test}(\X)\big\}$, whence also for any
$v,w\in L^0_\Cap(T\X)$ by approximation. This proves that $L^0_\Cap(T\X)$
is a Hilbert module, as required.
\end{proof}
\medskip

We now investigate  the relation that subsists between tangent $L^0(\Cap)$-module
and tangent $L^0(\mm)$-module. We start with the following result, which shows the existence of a natural projection operator $\bar\Pr$ sending $\bar\nabla f$ to $\nabla f$:
\begin{proposition}\label{pr:prtan}
There exists a unique linear continuous operator $\bar\Pr:\,L^0_\Cap(T\X)\to L^0_\mm(T\X)$
that satisfies the following properties:
\begin{itemize}
\item[$\rm i)$] $\bar\Pr(\bar\nabla f)=\nabla f$ for every $f\in{\rm Test}(\X)$.
\item[$\rm ii)$] $\bar\Pr(gv)=\Pr(g)\,\bar\Pr(v)$ for every $g\in L^0(\Cap)$
and $v\in L^0_\Cap(T\X)$.
\end{itemize}
Moreover, the operator $\bar\Pr$ satisfies
\begin{equation}\label{eq:preserves_bar_Pi}
\big|\bar\Pr(v)\big|=\Pr\big(|v|\big)\;\;\;\mm\text{-a.e.}
\quad\text{ for every }v\in L^0_\Cap(T\X).
\end{equation}
\end{proposition}
\begin{proof}
Given a Borel partition $(E_n)_{n\in\N}$ of $\X$ and
$(v_n)_{n\in\N}\subseteq L^0_\Cap(T\X)$, we are forced to set
\begin{equation}\label{eq:def_bar_Pi}
\bar\Pr\bigg(\sum_{n\in\N}[\nchi_{E_n}]_\Cap\bar\nabla f_n\bigg)
:=\sum_{n\in\N}[\nchi_{E_n}]_\mm\nabla f_n.
\end{equation}
The well-posedness of such definition stems from the following $\mm$-a.e.\ equalities:
\begin{equation}\label{eq:equality_bar_Pi}\begin{split}
\bigg|\sum_{n\in\N}{[\nchi_{E_n}]}_\mm\nabla f_n\bigg|
&=\sum_{n\in\N}{[\nchi_{E_n}]}_\mm\,|Df_n|
=\sum_{n\in\N}\Pr\big({[\nchi_{E_n}]}_\Cap\big)\,\Pr\big(\qcr\big(|Df_n|\big)\big)\\
&=\Pr\bigg(\sum_{n\in\N}{[\nchi_{E_n}]}_\Cap\,\qcr\big(|Df_n|\big)\bigg)
=\Pr\bigg(\sum_{n\in\N}{[\nchi_{E_n}]}_\Cap\,|\bar\nabla f_n|\bigg)\\
&=\Pr\bigg(\Big|\sum_{n\in\N}{[\nchi_{E_n}]}_\Cap\bar\nabla f_n\Big|\bigg).
\end{split}\end{equation}
Moreover, we also infer that such map $\bar\Pr$ -- which is linear by construction --
is also continuous, whence it admits a unique linear and continuous
extension $\bar\Pr:\,L^0_\Cap(T\X)\to L^0_\mm(T\X)$. Property i)
is clearly satisfied by \eqref{eq:def_bar_Pi}. From the linearity of $\nabla$
and $\bar\nabla$, we deduce that property ii) holds for any
simple function $g\in L^0(\Cap)$, thus also for any $g\in L^0(\Cap)$ by approximation.
Finally, again by approximation we see that \eqref{eq:preserves_bar_Pi} follows
from \eqref{eq:equality_bar_Pi}.
\end{proof}
The fact that $L^0_\Cap(T\X)$ can be thought of as a natural `refinement' of the already known $L^0_\mm(T\X)$ is now encoded in the following proposition, which shows that $\big(L^0_\mm(T\X),\bar\Pr\big)$ is the canonical quotient of $L^0_\Cap(T\X)$ up to $\mm$-a.e.\ equality (recall Proposition  \ref{pr:pr}):
\begin{proposition}\label{pr:prtan2}
Let $\mathscr N_\mm$ be a $L^0(\mm)$-module and $T:L^0_\Cap(T\X)\to \mathscr N_\mm$ linear and such that
\begin{equation}
\label{eq:ipT2}
\big|T(v)\big|\leq\Pr\big(|v|\big)\;\;\;\mm\text{-a.e.}\quad\text{ for every }
v\in L^0_\Cap(T\X).
\end{equation}
Then there is a unique $L^0(\mm)$-linear and continuous map $S:L^0_\mm(T\X)\to \mathscr N_\mm$ such that the diagram
\begin{equation}
\label{eq:diag2}
\begin{tikzpicture}[node distance=2.5cm, auto]
  \node (A) {${\rm Test}(\X)$};
  \node (B) [right of=A] {$L^0_\Cap(T\X)$};
  \node (C) [below  of=A] {$L^0_\mm(T\X)$};
  \node (D) [below  of=B] {$\mathscr N_\mm$};
  \draw[->] (A) to node {$\bar\nabla $} (B);
  \draw[->] (A) to node [swap] {$\nabla$} (C);
  \draw[->] (B) to node {$T$} (D);
  \draw[->] (C) to node [swap] {$S$} (D);
\end{tikzpicture}
\end{equation}
commutes.
\end{proposition}
\begin{proof} 
By \eqref{eq:ipT2} it follows that
$\big|T(\bar\nabla f)-T(\bar\nabla g)\big|\leq\big|D(f-g)\big|$ holds $\mm$-a.e.\ and thus
\begin{equation}\label{eq:comp}
f,g\in {\rm Test}(\X),\ \big|D(f-g)\big|=0\;\;\;\mm\text{-a.e.\ on }E\subset\X\quad\Longrightarrow\quad T(\bar\nabla f)=T(\bar\nabla g)\;\;\;\mm\text{-a.e.\ on }E.
\end{equation}
Now let $V\subset L^0_\mm(T\X)$ be the space of finite sums of the form $\sum_i[\nchi_{E_i}]_\mm\nabla f_i$ for $(E_i)_i$ Borel subsets of $\X$ and $(f_i)_i\subset {\rm Test}(\X)$, and define $S:V\to\mathscr N_\mm$ as:
\[
S(v):=\sum_i[\nchi_{E_i}]_\mm\,T(\bar\nabla f_i)\in\mathscr N_\mm\quad\text{ for every }
v=\sum_i[\nchi_{E_i}]_\mm\nabla f_i\in V.
\]
The implication in \eqref{eq:comp} grants that this is a good definition, i.e.\ the value of $S(v)$ depends only on $v$ and not on how it is written as finite sum of the form $\sum_i[\nchi_{E_i}]_\mm\nabla f_i$. It is clear that $S$ is linear and that, by \eqref{eq:ipT2} and item i) of Theorem \ref{thm:tg_mod}, it holds 
\begin{equation}
\label{eq:boundS}
\big|S(v)\big|\leq|v|\;\;\;\mm\text{-a.e.}\quad\text{ for every }v\in V.
\end{equation}
In particular, $S$ is 1-Lipschitz from $V$ (with the $L^0_\mm(T\X)$-distance) to $\mathscr N_\mm$. Since $V$ is dense in $L^0_\mm(T\X)$, $S$  can be uniquely extended to a continuous map -- still denoted by $S$ -- from $L^0_\mm(T\X)$ to $\mathscr N_\mm$. Clearly such extension is linear and, by \eqref{eq:boundS}, it also satisfies $S\big([\nchi_{E}]_\mm v\big)=[\nchi_{E}]_\mm S(v)$ (e.g.\ by mimicking the arguments used in the proof of Proposition \ref{pr:pr}). These two facts easily imply $L^0(\mm)$-linearity, thus showing existence of the desired map $S$. For uniqueness simply notice that the value of $S$ on the dense subspace $V$ of $L^0_\mm(T\X)$ is forced by the commutativity of the diagram in \eqref{eq:diag2}.
\end{proof}

\subsection{Quasi-continuity of Sobolev vector fields on
\texorpdfstring{$\sf RCD$}{RCD} spaces}\label{ss:QC_Sob_vector_fields}
Let $(\X,\sfd,\mm)$ be an ${\sf RCD}(K,\infty)$ space, for some $K\in\R$.
The aim of this conclusive subsection is to prove that any element of the
space $H^{1,2}_C(T\X)$ admits a quasi-continuous representative, in a
suitable sense. We begin with the definition of quasi-continuous vector field on $\X$:
\begin{definition}[Quasi-continuity for vector fields]
We define the set ${\rm Test\bar V}(\X)\subseteq L^0_\Cap(T\X)$ as
\begin{equation}
{\rm Test\bar V}(\X):=\bigg\{\sum_{i=0}^n\qcr(g_i)\,\bar\nabla f_i\;\bigg|
\;n\in\N,\,(f_i)_{i=1}^n,(g_i)_{i=1}^n\subseteq{\rm Test}(\X)\bigg\}.
\end{equation}
Then the space $\Cqcvf$ of \emph{quasi-continuous vector fields} on $\X$
is defined as the $\sfd_{L^0_\Cap(T\X)}$-closure of ${\rm Test\bar V}(\X)$
in $L^0_\Cap(T\X)$. It clearly holds that $\Cqcvf$ is a vector subspace
of $L^0_\Cap(T\X)$.
\end{definition}
\begin{proposition}\label{lem:|v|_qc}
Let $v\in\Cqcvf$ be given. Then it holds that $|v|\in\Cqc$.
\end{proposition}
\begin{proof}
First of all, if $v=\sum_{i=0}^n\qcr(g_i)\,\bar\nabla f_i\in{\rm Test\bar V}(\X)$ then
\[\begin{split}
|v|^2&=\sum_{i,j=0}^n\qcr(g_i)\,\qcr(g_j)\,\la\bar\nabla f_i,\bar\nabla f_j\ra\\
&=\sum_{i,j=0}^n\qcr(g_i)\,\qcr(g_j)\,
\frac{\big|\bar\nabla(f_i+f_j)\big|^2-|\bar\nabla f_i|^2-|\bar\nabla f_j|^2}{2}\\
&=\sum_{i,j=0}^n\qcr(g_i)\,\qcr(g_j)\,
\frac{\qcr\big(|D(f_i+f_j)|^2\big)-\qcr\big(|Df_i|^2\big)-\qcr\big(|Df_j|^2\big)}{2}
\in\Cqc.
\end{split}\]
For general $v\in\Cqcvf$ we proceed by approximation: chosen any sequence
$(v_n)_n\subset{\rm Test\bar V}(\X)$ such that $\sfd_{L^0_\Cap(T\X)}(v_n,v)\to 0$,
or equivalently $\sfd_\Cap\big(|v_n-v|,0\big)\to 0$, we have that $|v_n|\to|v|$
with respect to $\sfd_\Cap$, whence accordingly $|v|\in\Cqc$ by
Proposition \ref{prop:QC_complete}.
\end{proof}
\begin{proposition}\label{prop:inj_bar_Pi}
It holds that the map
$\bar\Pr\restr\Cqcvf:\,\Cqcvf\to L^0_\mm(T\X)$ is injective.
\end{proposition}
\begin{proof}
Let $v,w\in\Cqcvf$ be such that $\bar\Pr(v)=\bar\Pr(w)$.
In other words, we have that
\[\Pr\big(|v-w|\big)\overset{\eqref{eq:preserves_bar_Pi}}=
\big|\bar\Pr(v-w)\big|=0\quad\text{ holds }\mm\text{-a.e.\ in }\X,\]
whence Proposition \ref{prop:eqmmimplieseqCap} grants that
$|v-w|=0$ holds $\Cap$-a.e.\ in $\X$. This shows that $v=w$,
thus proving the claim.
\end{proof}

We are ready to state and prove the main result of the paper:
any element of $H^{1,2}_C(T\X)$ admits a quasi-continuous
representative in $\Cqcvf$.
This is a generalisation of Theorem \ref{thm:qcrepresentative}
to vector fields over an $\sf RCD$ space.
\begin{theorem}[Quasi-continuous representative of a Sobolev vector field]\label{thm:qcrvf}
Let us fix any ${\sf RCD}(K,\infty)$ space $(\X,\sfd,\mm)$, for some $K\in\R$.
Then there exists a unique map
\begin{equation}
\bar\qcr:\,H^{1,2}_C(T\X)\longrightarrow\Cqcvf
\end{equation}
such that $\bar\Pr\circ\bar\qcr:\,H^{1,2}_C(T\X)\to L^0_\mm(T\X)$
coincides with the inclusion $H^{1,2}_C(T\X)\subset L^0_\mm(T\X)$.
Moreover, $\bar\qcr$ is linear and $\big|\bar\qcr(v)\big|=\qcr\big(|v|\big)$
holds for every $v\in H^{1,2}_C(T\X)$.
\end{theorem}
\begin{proof}
Fix $v\in H^{1,2}_C(T\X)$. Pick $(\bar v_n)_n\subseteq{\rm Test\bar V}(\X)$ such
that $v_n:=\bar\Pr(\bar v_n)\to v$ in $W^{1,2}_C(T\X)$.
We know from Lemma \ref{lem:|v|_Sobolev} that $|v_n-v|\in W^{1,2}(\X)$ and
$\big|D|v_n-v|\big|\leq\big|\nabla(v_n-v)\big|_{\sf HS}$ $\mm$-a.e.\ for all $n\in\N$,
thus accordingly $|v_n-v|\to 0$ in $W^{1,2}(\X)$ as $n\to\infty$.
Proposition \ref{prop:linkqusob} grants that -- up to a (not relabeled) subsequence --
we have that $\qcr\big(|v_n-v|\big)\to 0$ locally quasi-uniformly as $n\to\infty$,
whence $\qcr\big(|v_n-v_m|\big)\to 0$ locally quasi-uniformly as $n,m\to\infty$.
Thus Proposition \ref{prop:QC_complete} yields
\[\sfd_{L^0_\Cap(T\X)}(\bar v_n,\bar v_m)
=\sfd_\Cap\big(\qcr\big(|v_n-v_m|\big),0\big)\longrightarrow 0
\quad\text{ as }n,m\to\infty.\]
This shows that $(\bar v_n)_n\subseteq L^0_\Cap(T\X)$ is Cauchy,
thus it converges to some $\bar v\in L^0_\Cap(T\X)$. Hence one has
$\bar\Pr(\bar v)=\bar\Pr\big(\lim_n\bar v_n\big)=\lim_n\bar\Pr(\bar v_n)
=\lim_n v_n=v$, so that we define $\bar\qcr(v):=\bar v$.
Proposition \ref{prop:inj_bar_Pi} grants that the map $\bar\qcr:\,H^{1,2}_C(T\X)\to\Cqcvf$
is well-defined and is the unique map such that $\bar\Pr\circ\bar\qcr$ coincides
with the inclusion $H^{1,2}_C(T\X)\subset L^0_\mm(T\X)$. Finally, the last
two statements follow from linearity of $\bar\Pr$,
Theorem \ref{thm:qcrepresentative} and Proposition \ref{prop:inj_bar_Pi}.
\end{proof}
\begin{remark}{\rm From the defining property of $\bar\qcr$ and Propositions \ref{pr:prtan}, \ref{prop:inj_bar_Pi} we see that $\bar\qcr(\nabla f)=\bar\nabla f$ for every $f\in {\rm Test}(\X)$. Then it is easy to see that   $\bar\qcr\big({\rm TestV}(\X)\big)={\rm Test\bar V}(\X)$.
\fr}\end{remark}

\begin{remark}[Alternative notion of quasi-continuous vector field]{\rm
It is well-known that a vector field $v:\,\R^n\to\R^n$ in the Euclidean space
is quasi-continuous if and only if $\R^n\ni x\mapsto\big|v(x)-\nabla f(x)\big|$
is quasi-continuous for every smooth function $f:\,\R^n\to\R$.
This would suggest an alternative definition of quasi-continuous vector field
on the ${\sf RCD}(K,\infty)$ space $\X$:
\begin{equation}
\label{eq:qtt}
\tCqcvf:=\big\{v\in L^0_\Cap(T\X)\;\big|\;
|v-\bar\nabla f|\in\Cqc\text{ for every }f\in{\rm Test}(\X)\big\}.
\end{equation}
The well-posedness of the previous definition follows from
the fact that quasi-continuity is preserved under modification on $\Cap$-negligible sets.
As we are going to show, it holds that
\begin{equation}\label{eq:C_subset_tC}
\Cqcvf\subset\tCqcvf.
\end{equation}
In order to prove it, let us fix $v\in\Cqcvf$. Given any $f\in{\rm Test}(\X)$ and
$(v_n)_n\subset{\rm Test\bar V}(\X)$ such that $v_n\to v$ in $L^0_\Cap(T\X)$,
we see (by arguing as in the proof of Lemma \ref{lem:|v|_qc}) that
$|v_n-\bar\nabla f|\in\Cqc$ holds for every $n\in\N$, therefore also
$|v-\bar\nabla f|\in\Cqc$ as an immediate consequence of the fact that
$|v_n-\bar\nabla f|\to|v-\bar\nabla f|$ in $L^0(\Cap)$. Since $f\in{\rm Test}(\X)$
is arbitrary, the claim \eqref{eq:C_subset_tC} is proven.

Notice that due to the non-linearity of the defining condition \eqref{eq:qtt} it is not clear if $\tCqcvf$ is a vector space or not. In particular, it is not clear if the inclusion in \eqref{eq:C_subset_tC} can be strict.
\fr}\end{remark}

We conclude giving a simple and explicit example to which the definitions and constructions presented in the paper can be applied:

\begin{example}[The case $\X=  {[0,1]}$]\label{ex:1d}{\rm Let us see how the definitions we gave work in the case $\X$ is the Euclidean segment $[0,1]$ equipped with its natural distance and measure. It is well known and easy to check that in this space every singleton has positive capacity. It follows that the space $L^0(\Cap)$ coincides, as a set, with the space of all real valued Borel functions on $\X$ and similarly the space $\Cqc$ coincides with the space of continuous functions on $\X$. In particular, the quasi-continuous representative of a Sobolev function is, in fact, its continuous representative. 

A direct verification of the definitions then shows that for $f\in D(\Delta)\subset W^{1,2}(\X)$ we have $f'\in W^{1,2}(\X)$ as well and, identifying $f,f'$ with their continuous representatives, it also holds $f'(0)=f'(1)=0$. In particular, for any $f\in {\rm Test}(\X)$ we have that (the continuous representative of) $f'$ is continuous and equal to 0 in $\{0,1\}$.

We now claim that $L^0_\Cap(T\X)$  is (=can be identified with) the space of Borel functions on $[0,1]$ which are 0 on $\{0,1\}$, the corresponding gradient map $\bar\nabla$ being the one which assigns to any $f\in {\rm Test}(\X)$ the continuous representative of $f'$, which shall hereafter be denoted by $f'$. The verification of this claim follows from the above discussion and the uniqueness part of Theorem \ref{thm:tg_mod}.

It is then clear that $\Cqcvf$ consists of continuous elements in $L^0_\Cap(T\X)$, i.e.\ of continuous functions which are zero on $\{0,1\}$, and that $\Cqcvf$ coincides with the space $\tCqcvf$ introduced in the previous remark.

This simple example shows that:
\begin{itemize}
\item[a)] The constant dimension property of $\RCD(K,N)$ spaces recently obtained in \cite{BS18}, which is known to carry over to the `standard' tangent module $L^0_\mm(T\X)$, does not carry over to the module $L^0_\Cap(T\X)$ introduced in this manuscript: adapting the definitions in \cite{Gigli14}, one can see that in our example the dimension of $L^0_\Cap(T\X)$ over $\{0,1\}$ is 0 and over $(0,1)$ is 1.
\item[b)] The estimates obtained in \cite{NV17} from which one can deduce that the capacity of the critical set of solutions of elliptic PDEs is 0, do not carry over to the $\RCD$ setting, and in fact not even in the setting of non-collapsed $\RCD$ spaces. Indeed, in our example the critical set of any function on $\X$ whose Laplacian is also a function (and not a measure) contains $\{0,1\}$: the problem seems to be the presence of the `boundary', see also \cite{GDP17} for further comments about the definition of boundary of a $\ncRCD$ space.   
\end{itemize} \ 
}\fr\end{example}

\end{document}